\documentclass[a4paper,reqno,11pt,oneside]{amsart}
\usepackage{amsmath,geometry}
\usepackage{graphicx}
\usepackage{mathrsfs}
\usepackage{amssymb,amsmath, amsfonts, amsthm}
\usepackage{latexsym}
\usepackage{color} 
\usepackage[dvips]{epsfig}
\usepackage[colorlinks]{hyperref}
\usepackage{tikz}
\usepackage{pgfplots}

\usepackage{bm}
\newtheorem{theorem}{Theorem}[section]
\newtheorem{lemma}[theorem]{Lemma}

\newtheorem{proposition}[theorem]{Proposition}

\newtheorem{rem}[theorem]{Remark}

\numberwithin{equation}{section}

\newcommand{\mf}[1]{\mathbf{#1}}
\renewcommand{\(}{\left(}
\renewcommand{\)}{\right)}
\newcommand{\R}{\mathbb R}

\newcommand{\beq}{\begin{equation}}
\newcommand{\eeq}{\end{equation}}

\newcommand{\intR}{\int_{\mathbb R^3}}

\begin{document}
\title[Segregated Solutions]{Segregated solutions for nonlinear Schr\"odinger systems with weak interspecies  forces}
 
\author{Angela Pistoia}
\address[A.Pistoia]{Dipartimento di Scienze di Base e Applicate per l'Ingegneria, Sapienza Universit\`a di Roma, Via Scarpa 16, 00161 Roma, Italy}
\email{angela.pistoia@uniroma1.it}

\author{Giusi Vaira}
\address[G.Vaira]
{Dipartimento di Matematica, Universit\`a degli studi di Bari ``Aldo Moro'', via Edoardo Orabona 4,70125 Bari, Italy}
\email{giusi.vaira@uniba.it}

\begin{abstract}
We find positive non-radial solutions for   a system of Schr\"odinger equations  in a weak fully  attractive or  repulsive regime in presence of an external radial trapping potential that exhibits a maximum or a minimum at infinity. 
\end{abstract}

\date\today
 \thanks{The authors are partially supported by INDAM-GNAMPA funds. A. Pistoia is also partially supported by Fondi di Ateneo ``Sapienza" Universit\`a di Roma (Italy). G. Vaira is also partially supported by PRIN 2017JPCAPN003 ``Qualitative and quantitative aspects of nonlinear PDEs" }

\maketitle

    \section{Introduction}

We are interested in finding positive solutions to  the system
\beq\label{S}  -\Delta u_i +\lambda_i u_i +V_i(x) u_i= \mu_i u_i^3+u_i   \sum \limits_{ j=1\atop j\not=i }^d \beta_{ij} u_j^2  \ \hbox{in}\ \mathbb R^n,\ i =1,\dots,d\eeq
 where $\mu_i >0$, $\lambda_i>0$, $ \beta_{ij} =\beta_{ji}\in\mathbb R$, $V_i\in C^0(\mathbb R^n)$, $d\in\mathbb N$, $n=2,3.$
This system   has been proposed as a mathematical model for multispecies
Bose-Einstein condensation in $m$ different  states:
\beq\label{GP}
- \iota \partial_t  \phi_i = \Delta  \phi_i -V_i(x) \phi_i+ \mu_i| \phi_i|^2 \phi_i+  \sum\limits_{j=1\atop j\not=i}^m\beta_{ij}
| \phi_j|^{2}  \phi_i ,\ i=1,\dots,d\eeq
where the complex valued functions $ \phi_i$'s are the wave functions of
the $i-$th condensate, $|\phi_i|$ is the amplitude of the $i-$th density, $ \mu_i$    describes  the interaction between particles of the same 
component
and $\beta_{ij}$, $i\not=j,$  describes   the interaction between particles   of different components, which can be {\em attractive} if $\beta_{ij}>0$ or {\em repulsive} if
  $\beta_{ij}<0$. 
To obtain solitary wave solutions of the Gross-Pitaevskii system \eqref{GP} we set  $ \phi_i(t,x) = e^{- \iota \lambda_i t} u_i(x)$ and we find real functions $u_i$'s which  solve the system
\eqref{S}. We refer to \cite{11,12,15,23} for a detailed physical motivation.\\

 There are different kind of solutions to \eqref{S}. The  trivial  solution has all trivial components, i.e.  $u_i\equiv0$ for any $i.$
   A  non-trivial  solution has some trivial components, i.e.  $u_i\equiv0$ for some $i$ and in this case the system \eqref{S} reduces to a system with a less number of components. The most interesting solutions  are the so-called {\em fully non-trivial} solutions whose components are    all non-trivial, i.e.  $u_i\not\equiv0$ for any $i$.
Among the set of fully non-trivial solutions we can distinguish among
 { \em synchronized} solutions and {\em non-synchronized} solutions.
  We say that ${\bf u}=(u_1,\dots,u_m)$ is a {\em  synchronized} solution to the system \eqref{S} (if, for example, $\lambda_i=\lambda$ and $V_i(x)=0$ for
  any index $i$),
if   ${\bf u}=( \gamma_1 U,\dots,  \gamma_kU)$ where $ \gamma_i>0$  and
$U$  is a positive solution to  the single equation  
\beq\label{se}-\Delta U+ \lambda  U =U^3\ \hbox{in}\ \mathbb R^n.\eeq
In this case  the system  \eqref{S} reduces   to the algebraic system
$$  \gamma_i=\mu_i  \gamma_i^3+ \gamma_i \sum\limits_{j=1\atop j\not=i}^m\beta_{ij} \gamma_j^{2},\ i=1,\dots,m.
$$ 
 
 Now, let us recall some known results. First, let us focus on the system with two components
 \beq\label{sis2c}\left\{\begin{aligned}
&-\Delta u +\lambda_1 u= \mu_{1}u^{3} +\beta u v^{ 2 }\ \hbox{in}\ \mathbb R^n, \\
&-\Delta v +\lambda_2 v= \mu_{2}v^3+  \beta u^2v\ \hbox{in}\ \mathbb R^ n
\end{aligned}\right.\eeq
It is immediate to check that if $\lambda=\lambda_1=\lambda_2$, the system \eqref{sis2c} has a synchronized solution
\beq\label{sincro2}
(u,v)= \(\gamma_1U,\gamma_2 U\),\ \gamma_1=\sqrt \frac{\beta-\mu_2}{\beta^2 - \mu_1 \mu_2},\ \gamma_2=\sqrt\frac{\beta-\mu_1}{\beta^2 - \mu_1 \mu_2} \eeq
where $U$ solves \eqref{se} provided
 \beq\label{sincrobeta}  -\sqrt{\mu_1 \mu_2} <\beta < \min\{\mu_1,\mu_2\}\ \hbox{or}\ \beta> \max\{\mu_1,\mu_2\}.\eeq
In the attractive case, i.e.  $\beta>0,$ all the positive solutions of \eqref{sis2c}
   are radially symmetric
(up to translation)  and both components are decreasing in the radial
variable (see for example \cite{Troy}).
On the other hand, in the repulsive case, i.e.  $\beta<0,$ there exists a large variety of solutions.  In fact radial solutions  has been found by 
Wei \& Weth \cite{WW2}, who proved that  if  $\beta\le-1$  for any  integer $k$ there exists a  radial solution  $(u_1,u_2)$ to the system
\eqref{sis2c} with $\lambda_i=1$ and $\mu_i=1$
such that the difference 
  $u_1-u_2$ has exactly $k-1$ zeroes and
   converges  as $\beta\to-\infty$ 
to a function $W$ which is radial changing-sign solutions of the scalar equation \eqref{se}.
 Bartsch, Dancer \& Wang \cite{BDW}  extend this result   to a larger
range of parameters $\beta,\mu_1,\mu_2.$  The case of an arbitrary number of components $d\geq3$ has been studied by    Terracini \& Verzini  \cite{tv}.
Problem  \eqref{sis2c} can also have
 non-radial  solutions. When the coupling parameter $\beta<0$  is small, Lin \& Wei \cite{LW} found solutions    with  one component peaking at the origin
and the other having a finite number of peaks on a $k$-polygon in the plane $\mathbb R^2$.
Wei \& Weth \cite{WW1} proved the existence of infinitely many non radial solutions which are  are invariant under
the action of a finite subgroup of $\mathcal O(n)$ (for example they satisfy \eqref{rot}).\\
  
The general case has been firstly considered by
Lin \& Wei \cite{LW2} where they study   the autonomous  system 
 \beq\label{sisg}
 -\Delta u _i+  \lambda_iu_i=   \mu_i u_i^3+\sum\limits_{j=1\atop j\not=i}^d \beta_{ij}u_j^{ 2 }\ \hbox{in}\ \mathbb R^n,  \ i=1,\dots,d
\eeq 
They  prove  the existence of a ground state solution whose components are positive, radially
symmetric and strictly decreasing when  all the $\beta_{ij}$'s are positive and the matrix $\(\beta_{ij}\)_{i,j=1,\dots, d}$ (here we set $\beta_{ii}=\mu_i$) is positively definite and also that
the ground state solution does not exist anymore if  all the $\beta_{ij}$'s are negative. A systematic  analysis of   system \eqref{sisg} under the assumption
that it admits mixed couplings and the components are organized into different  groups has been recently developed by Wei \& Wu \cite{WWu} .
The first result concerning existence of non-radial solutions for systems with 3 components goes back to   Lin \& Wei \cite{LW2} who proved the existence of 
non-radial solutions to  \eqref{sisg}
if all the coupling parameters are small and  repulsion    is  much stronger than  attraction.
Recently, Peng, Wang \& Wang  \cite{PWW}   consider system \eqref{sisg} in the case the repulsive couplings are small    and    obtain   solutions with some of
the components synchronized between them while being segregated with the
rest of the components.\\

All the previous results deal with the autonomous case. The non-autonomous case has been 
studied by
Peng \& Wang \cite{PW}, who considered the system
$$\left\{\begin{aligned}
&-\Delta u _1+  V_1(x)u_1= \mu_1 u_1^{3} +\beta u_1 u_2^{ 2 }\ \hbox{in}\ \mathbb R^n, \\
&-\Delta u_2+ V_2(x) u_2= \mu_2 u_2^3+  \beta u_1^2u_2 \ \hbox{in}\ \mathbb R^ n
\end{aligned}\right.$$
and found infinitely many non-radial positive solutions when the potentials
  $V_1$ and $V_2$  are radially symmetric and satisfy
\beq\label{ipov}V_i(x)\sim 1+{ \mathfrak v_\infty^i\over |x|^{ q _i}}\ \hbox{as}\ |x|\to\infty,\ i=1,2.\eeq
They build
synchronized  solutions (if  $ q _1< q _2,\   \mathfrak v_\infty^1>0$ or $ q _1> q _2,\ \mathfrak v_\infty^2>0$  and the coupling parameter $\beta$ satisfies \eqref{sincrobeta})
 and segregated solutions (if $ q _1= q _2,$ $  \mathfrak v_\infty^1,\mathfrak v_\infty^2>0$ and
   $ {\beta<\beta_0}$ for some $ \beta_0>0$). 
The synchronized solutions look like a sum of $k$ copies of the syncronized   solution  \eqref{sincro2}
 $$( u_1,u_2 )\sim  \(\gamma_1\sum\limits_{i=1}^k U(x-\rho\xi_i),\gamma_2 \sum\limits_{i=1}^k U(x-\rho\xi_i)\)\ \hbox{as}\ k\to\infty$$
where the peaks satisfy
$$\xi_i=\(\cos{2(i-1)\pi\over k},\sin{2(i-1)\pi\over k},0\),\ \hbox{where the radius}\ \rho\sim Rk\ln k\ \hbox{for some $R>0$}.$$
Here $U$  is the unique positive radial solution to the single equation
 \begin{equation}\label{pblim0}
-\Delta U + U = U^3\ \hbox{in}\  \mathbb R^n.
\end{equation}
On the other hand, the profile of each component of the segregated solutions looks like a sum of $k$   copies of the solution to \eqref{pblim0}
 $$( u_1,u_2)\sim \( \sum\limits_{i=1}^k U(x-\rho_1\xi_i), \sum\limits_{i=1}^k U(x-\rho_2\eta_i)\)\ \hbox{as}\ k\to\infty$$
where the peaks $\xi_i$ of the first component are as above, while  the peaks $\eta_i$ of the second component are nothing but the peaks of the first one rotated by an angle $\pi\over k$ and the radii $ \rho_i\sim R_ik\ln k$ for some $R_i>0$.
The proof of their result relies on the idea    by Wei \& Yan  \cite{WY} who found infinitely many solutions  (positive   if $ \mathfrak v_\infty>0$  and 
sign-changing  if $ \mathfrak v_\infty<0$) to  the single Schr\"odinger equation
$$-\Delta u+V(x) u=|u|^{p-1}u\ \hbox{in}\ \mathbb R^n$$
 when   the potential $V$ is radially symmetric and satisfy \eqref{ipo-v}.
\\

 We will focus on the existence of segregated solutions and we will ask a couple of  questions which naturally arise.

\begin{itemize}
\item[(Q1)]  Peng \& Wang's result holds when the system has only two components. Can we find  segregated solutions when the system has  at least three  components?

\item[(Q2)] Peng \& Wang's result holds when all the  $\mathfrak v^i_\infty$'s in \eqref{ipov} are positive.
Do there exist any solutions when some $\mathfrak v^i_\infty$'s are negative? 
\end{itemize}

%

We will give some  partial positive answers in the particular case when all the parameters $\mu_i$'s are equal to $1,$ all the  $\beta_{ij}$'s are equal to a real number $\beta$ and all the potentials
$V_i$'s coincide with the  radial potential $V\in  C^1(\mathbb R^n)$  which satisfies
   for some $\mathfrak v_\infty\in \mathbb R,$ $ \nu> 1$ and $\epsilon>0$  
\begin{equation}\label{ipo-v}
V(|x|)=1+\frac{\mathfrak v_\infty}{|x|^{\nu}}+\mathcal O\(\frac{1}{|x|^{\nu+\epsilon}}\)\quad \hbox{$C^1$-uniformly as}\ |x|\to+\infty,\end{equation} 
so that  the system \eqref{S} reduces to 
\beq\label{sis}
 -\Delta u_i +V (x) u_i=  u_i^3+\beta u_i\sum_{j=1\atop j\neq i}^{d} u_j^2\ \mbox{in}\,\mathbb R^n,\quad i=1, \ldots, d.\eeq 

Our main result is the following one.
\begin{theorem}\label{main}  
{Let $d\ge3 $ and $\nu>\frac2{d-2}.$ There exists   $k_0>0$ such that for any integer $k\ge k_0$ there exists $\beta_k>0$ such that 
for any $\beta\in(0,\beta_k)$ if  $\mathfrak v_\infty>0$  or  for any $\beta\in(-\beta_k,0)$ if  $\mathfrak v_\infty<0$ }
  the system \eqref{sis} has a positive solution $(u_1, \ldots, u_d) $ whose components satisfy
 $$u_i(x)\equiv u(\Theta_i x)$$ 
 where
$$   \Theta_i:= \left(\begin{matrix}\cos {2 (i-1)\pi\over  {d} k} &\sin {2(i-1)\pi\over  {d} k} & 0\\\\
-\sin{2(i-1)\over {d} k}&\cos{2(i-1)\pi\over  {d} k}&0\\\\
0&0&  I_{(n-2)\times(n-2)}\end{matrix}\right)$$
and
$$u_1(x)\sim  \sum\limits_{\ell=1}^k  U(x-\rho\xi_\ell)\ \hbox{as}\ k\to\infty$$
with
$$\xi_\ell=\(\cos{2(\ell-1)\pi\over k},\sin{2(\ell-1)\pi\over k},0\),\ \rho\sim Rk\ln k\ \hbox{as}\ k\to\infty\ \hbox{for some}\ R>0.$$
\end{theorem}

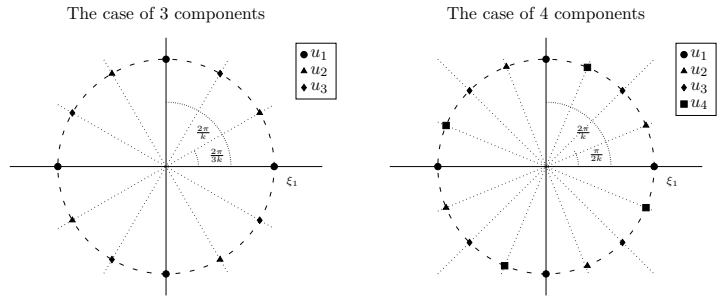
\begin{tikzpicture}[scale=0.6]\hskip2cm
\begin{axis} [axis equal, axis lines* = center,
xtick = \empty, ytick = \empty,legend style={anchor=north},title={The case of 3 components}]
]

\addplot[only marks,mark=*] coordinates {({cos(0)},{sin(0)})};
\addplot[only marks,mark=triangle*] coordinates {({cos(30)},{sin(30)})};
\addplot[only marks,mark=diamond*] coordinates {({cos(60)},{sin(60)})};
\addplot[only marks,mark=*] coordinates {({cos(90)},{sin(90)})};
\addplot[only marks,mark=triangle*] coordinates {({cos(120)},{sin(120)})};
\addplot[only marks,mark=diamond*] coordinates {({cos(150)},{sin(150)})};
\addplot[only marks,mark=*] coordinates {({cos(180)},{sin(180)})};
\addplot[only marks,mark=triangle*] coordinates {({cos(210)},{sin(210)})};
\addplot[only marks,mark=diamond*] coordinates {({cos(240)},{sin(240)})};
\addplot[only marks,mark=*] coordinates {({cos(270)},{sin(270)})};
\addplot[only marks,mark=triangle*] coordinates {({cos(300)},{sin(300)})};
\addplot[only marks,mark=diamond*] coordinates {({cos(330)},{sin(330)})};
\addplot
[domain=0:360,variable=\t,
samples=400,smooth,loosely dashed]
({cos(t)},{sin(t)});
\addplot[domain = -1:1, restrict y to domain =-1:1, samples = 400, dotted]{(sin(30))/(cos(30))*x};
\addplot[domain = -1:1, restrict y to domain =-1:1, samples = 400, dotted]{(sin(120))/(cos(120))*x};
\addplot[domain = -1:1, restrict y to domain =-1:1, samples = 400, dotted]{(sin(60))/(cos(60))*x};
\addplot[domain = -1:1, restrict y to domain =-1:1, samples = 400, dotted]{(sin(150))/(cos(150))*x};
 
 \addplot
[domain=0:30,variable=\t,
samples=400,smooth,densely dotted]
({0.3*cos(t)},{0.3*sin(t)});
\node[label={0:{\tiny$\textcolor{black}{2\pi\over 3k}$}}] at (axis cs: {0.3*cos(20)},{0.3*sin(20)}) {};

 \addplot
[domain=0:90,variable=\t,
samples=400,smooth,densely dotted]
({0.6*cos(t)},{0.6*sin(t)});
\node[label={270:{\tiny$ {2\pi\over k}$}}] at (axis cs: {0.6*cos(55)},{0.6*sin(55)}) {};
\node[label={330:{\tiny$ {\xi_1}$}}] at (axis cs: {1},{0}) {};

\legend{$u_1$,  $u_2$,   $u_3$}
\end{axis}

\hskip5cm

\begin{axis} [axis equal, axis lines* = center,
xtick = \empty, ytick = \empty,legend style={anchor=north},title={The case of 4 components}]
]

\addplot[only marks,mark=*] coordinates {({cos(0)},{sin(0)})};
\addplot[only marks,mark=triangle*] coordinates {({cos(22.5)},{sin(22.5)})};
\addplot[only marks,mark=diamond*] coordinates {({cos(45)},{sin(45)})};
\addplot[only marks,mark=square*] coordinates {({cos(67.5)},{sin(67.5)})};
\addplot[only marks,mark=*] coordinates {({cos(90)},{sin(90)})};
\addplot[only marks,mark=triangle*] coordinates {({cos(111.5)},{sin(111.5)})};
\addplot[only marks,mark=diamond*] coordinates {({cos(135)},{sin(135)})};
\addplot[only marks,mark=square*] coordinates {({cos(157.5)},{sin(157.5)})};
\addplot[only marks,mark=*] coordinates {({cos(180)},{sin(180)})};
\addplot[only marks,mark=triangle*] coordinates {({cos(202.5)},{sin(202.5)})};
\addplot[only marks,mark=diamond*] coordinates {({cos(225)},{sin(225)})};
\addplot[only marks,mark=square*] coordinates {({cos(247.5)},{sin(247.5)})};
\addplot[only marks,mark=*] coordinates {({cos(270)},{sin(270)})};
\addplot[only marks,mark=triangle*] coordinates {({cos(292.5)},{sin(292.5)})};
\addplot[only marks,mark=diamond*] coordinates {({cos(315)},{sin(315)})};
\addplot[only marks,mark=square*] coordinates {({cos(337.5)},{sin(337.5)})};

\addplot
[domain=0:360,variable=\t,
samples=400,smooth,loosely dashed]
({cos(t)},{sin(t)});
\addplot[domain = -1:1, restrict y to domain =-1:1, samples = 400, dotted]{(sin(22.5))/(cos(22.5))*x};
\addplot[domain = -1:1, restrict y to domain =-1:1, samples = 400, dotted]{(sin(45))/(cos(45))*x};
\addplot[domain = -1:1, restrict y to domain =-1:1, samples = 400, dotted]{(sin(67.5))/(cos(67.5))*x};
\addplot[domain = -1:1, restrict y to domain =-1:1, samples = 400, dotted]{(sin(111.5))/(cos(111.5))*x};
\addplot[domain = -1:1, restrict y to domain =-1:1, samples = 400, dotted]{(sin(135))/(cos(135))*x};
 \addplot[domain = -1:1, restrict y to domain =-1:1, samples = 400, dotted]{(sin(157.5))/(cos(157.5))*x};
 
 \addplot
[domain=0:30,variable=\t,
samples=400,smooth,densely dotted]
({0.3*cos(t)},{0.3*sin(t)});
\node[label={0:{\tiny$\textcolor{black}{\pi\over 2k}$}}] at (axis cs: {0.3*cos(20)},{0.3*sin(20)}) {};

 \addplot
[domain=0:90,variable=\t,
samples=400,smooth,densely dotted]
({0.6*cos(t)},{0.6*sin(t)});
\node[label={270:{\tiny$ {2\pi\over k}$}}] at (axis cs: {0.6*cos(55)},{0.6*sin(55)}) {};
\node[label={330:{\tiny$ {\xi_1}$}}] at (axis cs: {1},{0}) {};

\legend{$u_1$,  $u_2$,   $u_3$, $u_4$}
\end{axis}
\end{tikzpicture}

The proof of our result is given in Section \ref{princi}. First, using the symmetry, we write the system \eqref{sis} as a single  non-local equation \eqref{eqrid}. Next, we build a solution whose main term is the sum of a large number  of copies of solutions to problem \eqref{pblim0} (see \eqref{sol}) whose peaks are the vertices of a regular polygon with $k$ edges at distance $\rho=\rho(k)$ from the origin.  Due to the linear coupling term, the ansatz has to be improved by adding the solution of the linear problem \eqref{refin1}. Then we perform a classical Ljapunov-Schmidt procedure to reduce the problem to that of finding a radius $\rho$ which is the zero of the one-dimensional function \eqref{finale}.
This {\em reduced 1D function}   consists of  four  main terms (see Lemma \ref{crho}). The first term $\mathfrak v_\infty \frac k{\rho^{\nu+1}}$  arises from the potential effect and its sign depends on $\mathfrak v_\infty$, the second
term (which contains $  e^{-{2\pi\rho\over k}}$) is due to the interplay between peaks of the same component and   is always negative, the third term
(which contains  $\beta    e^{-{4\pi\rho\over d k}}$) is due to the interaction
among the peaks of different components
and its sign depends on the coupling parameter $\beta$.    If $d\ge3$ the third term prevails the second one, while if $d=2$ the second term dominates the third one. 
The asymptotic expansion  of the last term $\Upsilon$ which is produced by the correction of the ansatz is  really difficult to catch and, unfortunately, we believe  its presence   is not an innocence matter, since it could give a contribution to the second and the third terms.
We are only able to provide the   rough estimate \eqref{finale00}.  That is why 
we need to choose  the coupling parameter $\beta$ small enough so that $\Upsilon$ is an higher order term in the expansion of the reduced 1D function. 
It is clear that it would be extremely interesting to find the leading term of   $\Upsilon $. At this aim it is worthwhile  to point out that such an expansion
 relies on the asymptotic decay of the solution of \eqref{refin1} and this is the key point we are not able to solve.
\\
Now, let us go back to the reduced 1D function \eqref{finale}. We observe that Peng and Wang (using a different approach)   prove that  if the system has only two components the interaction among peaks of different components  is negligible with respect to the interaction among peaks of the same component and since such an interaction  is always negative (and does not depend on $\beta$), the existence of segregated solutions is ensured as soon as the potential $V$ has a minimum at infinity, i.e. $\mathfrak v_\infty>0.$ We conjecture that if the number of components is higher the interaction among peaks of different components should prevail. In such a case
the existence of  segregated solutions  would depend on   the behaviour of the potential $V$ at $\infty$ and the sign of the coupling parameter $\beta$, namely they should exist if either in an attractive regime (i.e. $\beta>0$) $V$ has a minimum at $\infty$    or in a repulsive regime (i.e. $\beta<0$) $V$ has a maximum at $\infty$ (i.e. $\mathfrak v_\infty<0$). Actually, this  is what happens when the coupling parameter $\beta$ is small as in Theorem \ref{main}. \\

It is also true for the  more general   variational system
\beq\label{sis-2d}
 -\Delta u_i +V (x) u_i=  |u_i|^{p-1}u_i+\beta|u_i|^{p-3\over 2}u_i\sum_{j=1\atop j\neq i}^{d} |u_j|^{p-1\over 2}u_j\ \mbox{in}\,\mathbb R^n,\quad i=1, \ldots, d,\eeq 
 when $p>3$ if $n=2$ or $p\in\(3,\frac{n+2}{n-2}\)$ if $n\ge3$ (the system \eqref{sis-2d}    with $p=3$ reduces to \eqref{sis}).
In fact, in Section \ref{rela} we show the following result.
\begin{theorem}\label{main1}  
Assume  \eqref{ipo-v} and
\begin{itemize}
\item[(i)] $d>\frac{p+1}2,$  $p>5$ and either $\mathfrak v_\infty>0 $ if $\beta>0$ or $\mathfrak v_\infty<0 $ if $\beta<0$
\item[(ii)] $d\le\frac{p-1}2$   and $\mathfrak v_\infty>0 .$ 
\end{itemize}
There exists   $k_0>0$ such that for any integer $k\ge k_0$  
  the system \eqref{sis} has a  solution $(u_1, \ldots, u_d) $ whose components satisfy the properties listed in Theorem \ref{main}.
\end{theorem}
The situation in this case is easier, because the correction of the ansatz is not needed and the reduced 1D function reads as \eqref{ckr}
in case (i) or \eqref{ckr2} in case (ii). To conclude, we highlight the fact that if we were able to approximate  the 1D function in Lemma \ref{crho}  as \eqref{ckr} or \eqref{ckr2}, then we would have the existence  of solutions of system \eqref{sis} without any assumptions on the smallness of the interspecies force $\beta.$
\\

The paper is organized as follows. Section \ref{princi} and Section \ref{rela} contain  the proof of Theorem \ref{main} and Theorem \ref{main1},
respectively. Appendix \ref{app1} contains some ausiliary results. Appendix \ref{app2} contains the proof of  Proposition \ref{inv}.\\

{\em Notation.} In what follows we agree that $f\lesssim g$  means $|f|\le c |g|(1+o(1))$ for some positive constant $c$ indipendent from $k$ and $f\sim g$  means $f= g(1+o(1))$.

\section{Proof of Theorem \ref{main}} \label{princi}

 \subsection{Reducing the system to a non-local equation}\label{1}
First of all, we introduce some symmetries which allow to  reduce the system \eqref{sis} to a non-local equation.

Given $\theta\in[0,2\pi]$, let  $\Theta_\theta:\mathbb R^n\to\mathbb R^n$ be the rotation of an angle $\theta$ in the first two components, i.e.
\begin{equation}\label{roto}\Theta_\theta:= \left(\begin{matrix}\cos \theta &\sin \theta & 0\\\\
-\sin\theta&\cos\theta&0\\\\
0&0&  I_{n-2\times n-2}\end{matrix}\right).\end{equation}

Now, let $k\ge 2$ and set  \begin{equation}\label{rotoi}\hat\Theta_i :=\Theta_{ {2\pi\over d k}(i-1)},\ \hbox{ for any } \ i=1, \ldots, d.\eeq Note that $\hat\Theta_1=I_{n \times n}.$\\

  We look for a solution to \eqref{sis}
 as $\mf u=(u_1, \ldots, u_d)\in (H^1(\mathbb R^n))^d$ whose components satisfy
\begin{equation}\label{comps}u_i(x)\equiv u(\hat\Theta_i x)\ \hbox{for any}\ i=1,\dots,d.\end{equation} 
It is immediate to check that such a $\mf u$  solves the system \eqref{sis} if and only if
  $u$ solves the non-local  equation
\beq\label{eqrid}
-\Delta u+V(x)u=u^3+\beta u  \sum_{j=2}^d u^2(\hat\Theta_j x), \quad\mbox{in}\,\, \mathbb R^n.\eeq

\subsection{The ansatz for the non-local equation}\label{2}
We look for a solution to \eqref{eqrid} in the space
\begin{equation}\label{acca} \mathscr H:=\left\{u\in W^{1,2}(\mathbb R^n)\,\,:\,\, u\,\hbox{satisfies} \,\eqref{pari}\,\hbox{and}\,\eqref{rot}\right\},\end{equation}
i.e.  
\beq\label{pari} u(x_1,\dots,x_i,\dots)=u(x_1,\dots,-x_i,\dots)\ \hbox{for any}\ i=2,\dots,n
\eeq and  
\beq\label{rot} u(x)=u\(\Theta_{ {2\pi\over   k}(h-1)} x\)\ \hbox{for any}\  h=1, \ldots, k\ \hbox{(see \eqref{roto})}.\eeq
We are going to find a solution of \eqref{eqrid} whose main term looks like
 \beq\label{sol}W_\rho(x):=\sum_{h=1}^k U_h(x) \ \hbox{and}\   U_h(x):=U(x-\rho\xi_h )\ \hbox{solves \eqref{pblim0}},\eeq
where the peaks satisfy 
\begin{equation}\label{points}
\xi_h :=\Theta_{ {2\pi\over   k}(h-1)}\xi_1,\  h=1,\ldots, k,\ \hbox{with}\ \xi_1 :=(1,0,0)\in \mathbb R^2\times\mathbb R^{n-2},\   
\end{equation}
and 
\begin{equation}\label{raggio}\rho\in \mathcal D_k:=\left[r_1k\ln k, r_2k\ln k\right]\ \hbox{for some}\ r_2>r_1>0. \end{equation}
It is useful to remind that  $U$  is the unique positive radial solution to \eqref{pblim0} and  decays exponentially together with its radial derivatives $U'$, i.e.
\begin{equation}\label{groundstate}
 \lim_{r\to+\infty} r^{\frac{n-1}{2}}e^{r} U (r)=\mathfrak u >0\ \hbox{and}\ \lim_{r\to+\infty}\frac{U'(r)}{U(r)}=-1,
\end{equation}
where $\mathfrak u$ is a positive constant depending on $n.$

It turns out that the error which comes from the interaction term 
$$ \beta W_\rho(x)\sum_{i=2}^d W_\rho^2(\hat\Theta_i x)$$
is too large and we need to refine the ansatz adding the function $\Psi_\rho=\beta Y_\rho$ where $Y_\rho\in\mathscr H$ solves
\begin{equation}\label{refin1}\begin{aligned}
\mathcal L(Y_\rho)= W_\rho \sum_{i=2}^d W^2_\rho(\hat\Theta_i\cdot)-\gamma_\rho\partial_\rho W_\rho\ \hbox{in}\ \mathbb R^n\end{aligned}
\end{equation}
where the linear operator $\mathcal L$ is defined in \eqref{lin} and
\begin{equation}\label{refin2}
\gamma_\rho:= \frac{\int\limits_{\mathbb R^n} W_\rho(x)\sum_{i=2}^d W^2_\rho(\hat\Theta_i(x)) \partial_\rho W_\rho(x)dx}{\int\limits_{\mathbb R^n}\(\partial_\rho W_\rho(x)\)^2dx}.
\end{equation}
The existence of $\Psi_\rho$
follows by Proposition \eqref{inv}.
\\

Then we will build a solution to   \eqref{eqrid} as
\beq\label{ansatz}u=W_\rho+  \Psi_\rho+\Phi,\ \hbox{with $W_\rho$ as in \eqref{sol} and $\Psi_\rho$ as in \eqref{refin1}}.\eeq
Moreover the higher order term $\Phi\in \mathscr H $   satisfies the orthogonality condition 
\beq\label{orto}\int\limits_{\mathbb R^n} \partial_\rho W_\rho \Phi=0.
\eeq

\subsection{Rewriting the single equation via the finite dimensional reduction method}\label{3}

It is useful to rewrite   problem \eqref{eqrid} in terms of $\Phi$, i.e.
\beq\label{pbri}\mathcal L(\Phi)=\mathcal N(\Phi)+\mathcal E\ \hbox{in}\ \mathbb R^n,\eeq
where 
the linear operator $\mathcal L$ is defined by
\beq\label{lin}
\mathcal L(\Phi)=-\Delta\Phi+V(x)\Phi-3W_\rho^2\Phi   {-\beta \Phi\sum_{i=2}^d W^2_\rho(\hat\Theta_i x)-2\beta W_\rho\sum_{i=2}^d W _\rho(\hat\Theta_i x)\Phi(\hat\Theta_i x)},\eeq
the error term $\mathcal E$ is defined by
\beq\label{err}\begin{aligned}
\mathcal E:=&(1-V(x)) W_\rho  \\ &+\Delta W_\rho-W_\rho+W_\rho^3    \underbrace{-\mathcal L (\Psi_\rho)+ \beta W_\rho\sum_{i=2}^d W_\rho^2(\hat\Theta_i x)}_{=\gamma_\rho\partial_\rho W_\rho\ \hbox{(see \eqref{refin2})}}+ \Psi_\rho^3+3 W_\rho\Psi_\rho^2\\
&+ \beta \Psi_\rho\sum_{i=2}^d  2W_\rho(\hat\Theta_i x) \Psi_\rho(\hat\Theta_i x) +\beta \Psi_\rho\sum_{i=2}^d \Psi^2_\rho(\hat\Theta_i x) 
+\beta W_\rho\sum_{i=2}^d    \Psi^2_\rho(\hat\Theta_i x)
\end{aligned}\eeq
and the higher order term $\mathcal N(\Phi)$ is defined by
\beq\label{qua}\begin{aligned}
\mathcal N(\Phi)&:=  \Phi^3+3W_\rho\Phi^2+3\Psi_\rho\Phi^2+3\Psi_\rho^2\Phi+6 W_\rho\Psi_\rho\Phi\\
&+\beta \Phi\sum_{i=2}^d \left[2W_\rho(\hat\Theta_i x)\(\Psi_\rho (\hat\Theta_i x)+\Phi (\hat\Theta_i x)\)+ \(\Psi_\rho (\hat\Theta_i x)+\Phi (\hat\Theta_i x)\)^2\right]\\&+\beta \Psi_\rho\sum_{i=2}^d 2\Phi(\hat\Theta_i x)\(W_\rho(\hat\Theta_i x)+ \Psi_\rho (\hat\Theta_i x)\)+{\beta \Psi_\rho\sum_{i=2}^d \Phi^2(\hat\Theta_i x)}
\\
&+\beta W_\rho\sum_{i=2}^d \Phi^2(\hat\Theta_i x)+2\beta W_\rho\sum_{i=2}^d  \Psi_\rho (\hat\Theta_i x)\Phi(\hat\Theta_i x).\end{aligned}\eeq

In order to solve \eqref{pbri} we use the classical  Lyapunov-Schmidt procedure:
\begin{itemize}
\item[(i)] first,  given   $\rho\in\mathcal D_k$ (see \eqref{raggio}) we find a function $\Phi\in \mathscr H$   such that, for a certain $\mathfrak c \in\mathbb R$, it solves the {\em intermediate nonlinear problem}
\beq\label{pbint}
 \mathcal L(\Phi)=\mathcal N(\Phi)+\mathcal E+\mathfrak c \partial_\rho W_\rho\ \hbox{in}\ \mathbb R^n\ \hbox{and}\   \int\limits_{\mathbb R^n} \partial_\rho W_\rho \Phi=0.\eeq 
\item[(ii)] next, we   find $\rho\in\mathcal D_k$ (see \eqref{raggio}) so that $\mathfrak c$ in \eqref{pbint} is equal to zero.\end{itemize}

 \subsection{The linear theory}\label{4}
 We will find the solution $\Phi$ to \eqref{pbint} in a suitable   Banach space where    the linear operator $\mathcal L$   defined in \eqref{lin}
 is invertible.\\
 
 We introduce the Banach space
 (see also \cite[Section 3]{mpw})
 $$\mathscr B_{*}:=\{h\in L^\infty(\mathbb R^n)\,\,:\,\, \|h\|_{*}<+\infty\},\ 
\|h\|_{*}:=\sup_{x\in\mathbb R^n}\left(\sum_{i=1}^d\sum_{j=1}^k e^{-\alpha|x-\rho\eta_{ij}|}\right)^{-1}|h(x)|.$$ 
for some $\alpha\in(0, 1)$. The points
  $\eta_{i\ell}:=\hat\Theta_i^{-1}\xi_\ell $ are  the peaks of the bubble $ U\(\hat\Theta_i \cdot-\rho\xi_\ell\),$ namely
$ U\(\hat\Theta_i x-\rho\xi_\ell\)= U\(x-\rho\hat\Theta_i^{-1}\xi_\ell\)$.\\
It is worthwhile to point out that all the peaks $\eta_{ij}$'s are different among them.
 \begin{lemma}\label{dist-punti}
It holds true 
\begin{equation}\label{min1}\min\limits_{h=2,\dots,k}\left|\xi_1 - \xi_h\right|=\left|\xi_1 - \xi_2 \right|= 2\sin{\pi\over  k}\end{equation}
and
\begin{equation}\label{min2}\min\limits_{i=2,\dots,d\atop h,\ell=1,\dots,k}\left|\xi_h -\eta_{i\ell} \right|= \left|\xi_1 -\eta_{21} \right|=2\sin{\pi\over dk}.\end{equation}
\end{lemma}
\begin{proof}
\eqref{min1} is immediate. Let us check \eqref{min2}.
By \eqref{rotoi} and \eqref{rot}
$$\eta_{i\ell} =\Theta_i^{-1}\xi_\ell = \Theta_{\frac{2\pi}{k}(\ell-1)-\frac{2\pi}{dk}(i-1)}\xi_1\ \hbox{and}\ \xi_h=\Theta_{\frac{2\pi}{k}(h-1)}\xi_1$$
and they coincide if and only if
 $$\frac{2\pi}{dk}(i-1)= \frac{2\pi}{k}(\ell-h)\ \iff\  i= d(\ell-h)+1\ \iff\ i=1 \ \hbox{and}\ \ell=h. $$ 
Moreover, when $i\ge2$ the distance $|\xi_h -\eta_{i\ell}|$ is minimal when the angle $$\frac{2\pi}{dk}(i-1)-\frac{2\pi}{k}(\ell-h)=\frac{2\pi}{k}\(\frac 1 d (i-1)-(\ell-h)\)$$ is minimal, that is if  $\ell=h$ and $i=2$ and in this case $$ |\xi_h -\eta_{ih}|= |\xi_1 -\eta_{21}|=2 \sin \frac{\pi}{d k} .$$
\end{proof}
 
It is also  useful to remark that
\beq\label{normastella}\|h\|_{L^\infty}{\le C} \|h\|_{*}\quad \mbox{and}\quad \|h\|_{L^q}{\le C}\|h\|_{*},\quad \forall\,\, 1\le q<\infty.\eeq   
Actually,  the linear operator $\mathcal L$   defined in \eqref{lin}
 is invertible in $\mathscr B_{*}$ as shown in the following proposition whose proof can be obtained arguing as in  \cite[Section 3]{dwy} and   is postponed in the Appendix \ref{app2}.

\begin{proposition}\label{inv}
{For any compact set $B\subset \mathfrak B$  (see Remark \ref{wep}), there exists $k_0>0$ such that for any $\beta\in B$, for any $k\ge k_0$,  $\rho\in\mathcal D_k$ and   $ h \in \mathscr B_{*}$ which satisfy \eqref{pari} and \eqref{rot},} the linear problem 
$$  \mathcal L(\Phi)= h +\mathfrak c \partial_\rho W_\rho\ \hbox{in}\  \mathbb R^n\ \hbox{and}\ \int\limits_{\mathbb R^n} \partial_\rho W_\rho \Phi=0$$
admits a unique solution $\Phi=\Phi(\rho,k)\in\mathscr H\cap\mathscr B_{*}$ and $\mathfrak c=\mathfrak c(\rho,k)\in\mathbb R$ such that
\beq\label{norma}
\|\Phi\|_{*}\lesssim\| h \|_{*}\quad\mbox{and}\quad |\mathfrak c|\lesssim\| h \|_{*}.\eeq
\end{proposition}

\begin{rem}\label{wep} Let $\Lambda_\kappa$ be the sequence of eigenvalues of the problem
$$-\Delta\varphi+\varphi=\Lambda_\kappa U^2 \varphi\ \hbox{in}\ \mathbb R^n.$$ 
It is well known that $\Lambda_1=1$ is simple and the associated eigenfunction is $ U.$ The second  eigenvalue $\Lambda_2$ has multiplicity $n$ and the   associated eigenfunctions  are $ \partial_{x_i}U$, $i=1,\dots,n.$ 
{Set $\mathfrak B:=\mathbb R\setminus\{\Lambda_\kappa,\ \kappa\in\mathbb N\}.$}\end {rem}

\subsection{The weighted norm}\label{5-}

We introduce the sector
$$\Sigma:=\left\{(r\cos\theta,r\sin\theta,x_3)\ :\ x_3\in\mathbb R,\ r\ge0,\ \theta\in\left[-\frac\pi k,\frac\pi k\right]\right\}$$  
so that if  $u$ satisfies   \eqref{rot}
then
 $$\begin{aligned}\|u\|_{*}& =\sup_{x\in\Sigma}\left(\sum_{j=1}^k\sum_{i=1}^d e^{-\alpha|x-\rho\eta_{ij}|}\right)^{-1}|u(x)|.\end{aligned}$$
Moreover,  it is useful to decompose $\Sigma$
into $d$ sectors (if $d$ is odd $i\in I:=\{-\frac{d-1}2,\dots,0,\dots,\frac{d-1}2\}$) or  in $d+1$ sectors (if $d$ is even and $i\in I:=\{-\frac d2,\dots,0,\dots,\frac d2\}$)
$$\widetilde \Sigma_i:=\left\{(r\cos\theta,r\sin\theta,x_3)\in\Sigma\ :\ x_3\in\mathbb R,\ r\ge0,\ \theta\in\left[-\frac\pi {dk}+\frac{2\pi}{dk}i,\frac\pi {dk}+\frac{2\pi}{dk}i\right]\right\},\ i\in I,$$
so that each sector $\widetilde \Sigma_i$ contains only one point $\eta_{\ell\kappa}$ which will be denoted by $\tilde\eta_i$. In particular $\tilde \eta_0=\xi_1.$ 
Therefore, to estimate the weighted norm $\|u\|_{*}$ it will be enough to select the points $\tilde\eta_i$ which belong to the sector $\Sigma$, i.e. 
$$\begin{aligned}\|u\|_{*}& =\sup_{x\in\Sigma}\left(\sum_{j=1}^k\sum_{i=1}^d e^{-\alpha|x-\rho\eta_{ij}|}\right)^{-1}|u(x)| \le \sup_{x\in\Sigma}    \left(\sum_{i\in  I} e^{-\alpha|x-\rho\tilde\eta_{i}|}\right)^{-1} |u(x)|.\end{aligned}$$
Moreover, since
\begin{equation}\label{si}x\in\Sigma\ \Rightarrow\ |x-\rho\xi_h|\ge\frac\rho2|\xi_h-\xi_1|\ \hbox{if}\ h\ge2\end{equation}
and
\begin{equation}\label{sii}x\in\widetilde\Sigma_i\ \Rightarrow\ |x-\rho\tilde\eta_{\ell}|\ge\frac\rho2|\tilde\eta_{\ell}-\tilde\eta_{i}|\ \hbox{if}\ \ell\not=i,\end{equation}
  in the sector $\Sigma$ the main term of $W_\rho(x)=\sum _{j=1}^k U(x-\rho\xi_j)$ is the first bubble $U(x-\rho\xi_1)$ whose peak lies in $\Sigma$, while in each subsector $ \widetilde\Sigma_\ell$ the main term of $W_\rho(\Theta_i x)=\sum _{j=1}^k U(x-\rho\Theta^{-1}_i\xi_j)$ is the bubble
$U(x-\rho\tilde\eta_\ell)$ whose peak $\eta_{ij}=\Theta^{-1}_i\xi_j=\tilde\eta_\ell$ is the unique which belongs to $ \widetilde\Sigma_\ell.$
This can be made rigorous in the following lemma (whose proof can be found in   \cite[Lemma A.1]{WY}).
\begin{lemma}
\label{chiave}
For any $\sigma\in (0,1)$ 
for any $j\ge2$
$$U(x-\rho\xi_j)\lesssim e^{-\sigma\frac{\pi\rho}k}e^{-(1-\sigma)|x-\rho\xi_1|}\ \hbox{for any}\ x\in\Sigma$$
and for any $\eta_{ij}\not\in  \widetilde\Sigma_\ell$
$$U(x-\rho\eta_{ij})\lesssim e^{-\sigma\frac{\pi\rho}{dk}}e^{-(1-\sigma)|x-\rho\tilde\eta_\ell|}\ \hbox{for any}\ x\in \widetilde\Sigma_\ell.$$
\end{lemma}

\subsection{The size of the error}\label{5}

First, by  Proposition \eqref{inv}, Lemma \ref{inter} and Lemma \ref{gamma} the function $\Psi_\rho=\beta Y_\rho in \mathscr H$ which solves
 \eqref{refin1} satisfies
\begin{equation}\label{refin4}
\|\Psi_\rho\|_{*}=\beta\|Y_\rho\|_{*}\lesssim \beta \|W_\rho\sum_{i=2}^d W^2_\rho(\hat\Theta_i \cdot)\|_{*}\lesssim { e^{-(1-\alpha)\frac{2\pi\rho}{d k}}}.\eeq

\begin{lemma}\label{inter}
 If $\alpha\in\(0,1\)$ it holds true that
\beq\label{inter1}
 \left\|W_\rho\sum_{i=2}^d W^2_\rho(\hat\Theta_i \cdot)\right\|_{*}\lesssim  e^{-(1-\alpha)\frac{2\pi\rho}{d k}}\eeq
and  \beq\label{inter2} \left\| W_\rho (\cdot)\sum_{i=2}^d W_\rho(\hat\Theta_i \cdot)\right\|_{*}\lesssim e^{-(1-\alpha)\frac{2\pi\rho}{dk}}.\eeq

\end{lemma}
\begin{proof}
{ We only prove \eqref{inter1}, because the proof of  \eqref{inter2} is similar. \\
Let $x\in \Sigma.$ There exists $\ell\in I$ such that $x\in \widetilde\Sigma_\ell$ and so
$$\begin{aligned}&\left(\sum_{j\in I}e^{-\alpha|x-\rho\tilde\eta_{j}|}\right)^{-1}|W_\rho(x)\sum_{i=2}^d W_\rho^2(\hat\Theta_ix)|\\ &=\left(e^{-\alpha|x-\rho\tilde\eta_{\ell}|}+\sum_{j\in I\atop \not=\ell}e^{-\alpha|x-\rho\tilde\eta_{j}|}\right)^{-1}|W_\rho(x)\sum_{i=2}^d W_\rho^2(\hat\Theta_ix)|\\ &\le
 e^{\alpha|x-\rho\tilde\eta_{\ell}|} |W_\rho(x)\sum_{i=2}^d W_\rho^2(\hat\Theta_ix)|.\end{aligned} $$
Now, by Lemma \ref{chiave}
$$\begin{aligned}W_\rho(x)\sum_{i=2}^d W_\rho^2(\hat\Theta_i x)&\lesssim U(x-\rho\tilde\eta_0)\sum_{i\in I\atop i\neq 0}U^2(x-\rho\tilde\eta_i)\\ &\lesssim e^{-\alpha|x-\rho\tilde\eta_\ell|}\sum_{i\in I\atop i\neq 0}e^{-|x-\rho\tilde\eta_0|-2|x-\rho\tilde\eta_i|+\alpha|x-\rho\tilde\eta_\ell|}\\ &\lesssim e^{-\alpha|x-\rho\tilde\eta_\ell|}e^{-(1-\alpha)\frac{2\pi\rho}{dk}}.\end{aligned}$$ 
Indeed if $\ell=0$  
$$\begin{aligned}|x-\rho\tilde\eta_0|+2|x-\rho\tilde\eta_i|-\alpha|x-\rho\tilde\eta_\ell|&= (1-\alpha)|x-\rho\tilde\eta_0|+2|x-\rho\tilde\eta_i|\\ &\ge (1-\alpha)\rho|\tilde\eta_i-\tilde\eta_0|-(1-\alpha)|x-\rho\tilde\eta_i|+2|x-\rho\tilde\eta_i|\\ &\ge (1-\alpha)\rho\min_{i\in I\atop i\neq 0}|\tilde\eta_i-\tilde\eta_0|\end{aligned}$$
and if  $\ell\neq 0$ since $|x-\rho\tilde\eta_0|\ge |x-\rho\tilde\eta_\ell|$  
$$\begin{aligned}|x-\rho\tilde\eta_0|+2|x-\rho\tilde\eta_i|-\alpha|x-\rho\tilde\eta_\ell|&\ge  (1-\alpha)|x-\rho\tilde\eta_0|+2|x-\rho\tilde\eta_i|\\ &\ge (1-\alpha)\rho|\tilde\eta_i-\tilde\eta_0|-(1-\alpha)|x-\rho\tilde\eta_i|+2|x-\rho\tilde\eta_i|\\ &\ge (1-\alpha)\rho\min_{i\in I\atop i\neq 0}|\tilde\eta_i-\tilde\eta_0|.\end{aligned}$$
Finally, the claim follows.}

\end{proof}

\begin{lemma}\label{gamma}
It holds true that
$$  \int\limits_{\mathbb R^n} W_\rho(x)\sum_{i=2}^d W_\rho^2(\hat \Theta_i x) \partial_\rho W_\rho(x)dx=\left\{\begin{aligned}&-    \mathtt C \(\frac k\rho\)^2e^{-4\rho\frac{\pi}{d k}}\ln\ln k+h.o.t.& \hbox{if}\ n=3\\
&-    \mathtt C \sqrt{\frac k\rho} e^{-4\rho\frac{\pi}{d k}}+h.o.t.& \hbox{if}\ n=2\end{aligned}\right. $$
\end{lemma}
for some positive constant $\mathtt C$.
\begin{proof} We prove the claim when $n=3.$ The proof in the case $n=2$ is similar.
By Lemma \ref{ruiz}, Lemma \ref{new} and Lemma \ref{chiave} and setting $\eta_{ij}:=\hat\Theta_i^{-1}\xi_j $ and $U_{ij}(x) = U(x-\rho\eta_{ij})$

$$\begin{aligned} &\intR\(\sum_{h=1}^k U_h\)\sum_{i=2}^d \(\sum_{j=1}^k U_j(\hat \Theta_i x)\)^2\partial_\rho W_\rho\\ &= k\int_{ \Sigma}\(U_1+\sum_{h=2}^k U_h\)\sum_{i=2}^d\(\sum_{j=1}^k U _{ij}\)^2\(\partial_\rho U_1+\sum\limits_{i=2}^k \partial_\rho U_i\)\\
&= k \int_{ \Sigma} U_1\partial_\rho U_1\sum_{i=2}^d  \sum_{j=1}^kU_{ij}^2+h.o.t. \\
&= k \sum_{\ell\in \mathcal I\atop\ell\not=0}\int_{ \Sigma} U(x-\rho\xi_1) U'(x-\rho\xi_1 ) {\langle x-\rho\xi_1,-\xi_1\rangle\over|x-\rho\xi_1|}   U^2(x-\rho\tilde\eta_\ell)dx+h.o.t.\\
&= k\sum_{\ell\in \mathcal I\atop\ell\not=0}\int_{ \mathbb R^3} \underbrace{U(x+\rho(\tilde\eta_\ell- \xi_1)) U'(x+\rho(\tilde\eta_\ell- \xi_1)  ) {\langle x+\rho(\tilde\eta_\ell- \xi_1),-\xi_1\rangle\over|x+\rho(\tilde\eta_\ell- \xi_1) |} }_{=-\langle\frac12\nabla_\zeta \Gamma_{22}(\zeta),\xi_1\rangle,\ \zeta=\rho\(\tilde\eta_\ell-\xi_1 \)}
 U^2(x)dx+h.o.t.\\
&= - k \mathfrak c\sum_{\ell\in \mathcal I\atop\ell\not=0}\frac12\left\langle  {\xi_1-\tilde\eta_\ell\over |\xi_1-\tilde\eta_\ell|},\xi_1\right\rangle {e^{-2\rho|\xi_1 -\tilde\eta_\ell|}\ln(\rho|\xi_1 -\tilde\eta_\ell|) \over  |\rho\(\xi_1 -\tilde\eta_\ell\)|^2  }+h.o.t.\\ & = 
-     \mathfrak c  \frac d{8\pi} k^2 \frac{e^{-4\rho\frac{\pi}{d k}}}{  \rho^2 }\ln\ln k+h.o.t. \end{aligned}$$
because of \eqref{min2}, the choice of $\rho\in \mathcal D_k$ \eqref{raggio} and the fact that
\beq\label{fino2}\min_{\ell\in \mathcal I\atop\ell\not=0}|\tilde\eta_\ell -\xi_1|=|\eta_{21}-\xi_1|\sim {2\pi\over dk} \ \hbox{and}\ \langle  \xi_1- \eta_{21},\xi_1\rangle=1-\cos{2\pi\over dk}\sim \frac12 \(2\pi\over dk\)^2.\eeq

\end{proof}
Finally, we find the size of the error.
\begin{lemma}\label{errore}
For any $\alpha \in (0, 1)$, it holds true that
$$\|\mathcal E\|_{*}\lesssim \frac{1}{\rho^\nu}+{ e^{-(1-\alpha)\frac{4\pi\rho}{d k}}}.$$
\end{lemma}
\begin{proof}
By \eqref{err} taking into account that 
  the functions $W_\rho,$ $\partial_\rho W_\rho$ and the weight in the norm are bounded  
we get
$$\begin{aligned}
\|\mathcal E\|_{*}&\lesssim\underbrace{\|\(1-V(x)\)W_\rho  \|_{*}}_{:=E_1}+\underbrace{\|\Delta W_\rho-W_\rho +W_\rho^3\|_{*}}_{:=E_2}+   \underbrace{|\gamma_\rho|}_{:=E_3} \\
&+ \underbrace{\| \Psi_\rho\|_{*}^3+ \| \Psi_\rho\|_{*}^2}_{\lesssim { e^{-(1-\alpha)\frac{4\pi\rho}{d k}}}\ \hbox{\tiny because of \eqref{refin4} }}\\
&\lesssim \frac 1{\rho^\nu}+{ e^{-(1-\alpha)\frac{4\pi\rho}{d k}}}.
\end{aligned}$$
 
\begin{itemize}
\item {\em Estimate of $E_1.$}
{Let $x\in\Sigma$ then there is $\ell\in I$ such that $x\in \widetilde\Sigma_\ell$. Then $$\left(\sum_{j\in I}e^{-\alpha|x-\rho\tilde\eta_j|}\right)^{-1}|E_1(x)|\lesssim e^{\alpha|x-\rho\tilde\eta_\ell|}|V(x)-1|e^{-|x-\rho\tilde\eta_0|}.$$ We note that if $|x|<\frac\rho 2$ then $|x-\rho\tilde\eta_0|\ge \frac\rho 2$. Now if $x\in  \widetilde\Sigma_0\cap \left\{|x|<\frac\rho 2\right\}$ then $$-\alpha|x-\rho\tilde\eta_\ell|+|x-\rho\tilde\eta_0|=(1-\alpha)|x-\rho\tilde\eta_0|\ge (1-\alpha)\frac\rho 2.$$ If $x\in \widetilde\Sigma_\ell\cap\left\{|x|<\frac\rho 2\right\}$ with $\ell\neq 0$ then $|x-\rho\tilde\eta_0|\ge |x-\rho\tilde\eta_\ell|$ and hence $$\begin{aligned}-\alpha|x-\rho\tilde\eta_\ell|+|x-\rho\tilde\eta_0|&=(1-\alpha)|x-\rho\tilde\eta_0|+\alpha\left(|x-\rho\tilde\eta_0|-|x-\rho\tilde\eta_\ell|\right)\\ &\ge (1-\alpha)|x-\rho\tilde\eta_0|\ge (1-\alpha)\frac\rho 2.\end{aligned}$$ Then, since $V\in L^\infty(\mathbb R^n)$ we get that in $\left\{|x|<\frac\rho 2\right\}$ $$\left(\sum_{j\in I}e^{-\alpha|x-\rho\tilde\eta_j|}\right)^{-1}|E_1(x)|\lesssim e^{-(1-\alpha)\frac\rho 2}.$$ 
If $|x|\ge \frac\rho 2$ by \eqref{ipo-v} $|V(x)-1|\lesssim \frac{1}{\rho^\nu}$ and so for $x\in  \widetilde\Sigma_0\cap \left\{|x|\ge \frac\rho 2\right\}$ $$-\alpha|x-\rho\tilde\eta_\ell|+|x-\rho\tilde\eta_0|=(1-\alpha)|x-\rho\tilde\eta_0|\ge 0.$$ For $x\in \widetilde\Sigma_\ell\cap \left\{|x|\ge \frac\rho 2\right\}$ with $\ell\neq 0$ since $|x-\rho\tilde\eta_0|\ge |x-\rho\tilde\eta_\ell|$ we get $$-\alpha|x-\rho\tilde\eta_\ell|+|x-\rho\tilde\eta_0|=(1-\alpha)|x-\rho\tilde\eta_0|+\alpha\left(|x-\rho\tilde\eta_0|-|x-\rho\tilde\eta_\ell|\right)\ge 0.$$ Hence in $\left\{|x|\ge \frac\rho 2\right\}$ we get
$$\left(\sum_{j\in I}e^{-\alpha|x-\rho\tilde\eta_j|}\right)^{-1}|E_1(x)|\lesssim \frac{1}{\rho^\nu}.$$ This implies $$\|E_1\|_{*}\lesssim \frac{1}{\rho^\nu}.$$
 }


\item {\em Estimate of $E_2.$} {Let $x\in\Sigma$ then there is $\ell\in I$ such that $x\in  \widetilde\Sigma_\ell$. }
Now since
$$\Delta W_\rho-W_\rho +W_\rho^3=\(\sum\limits_{h=1}^k U_h\)^3-\sum\limits_{h=1}^k U_h^3$$ we get that 
$${\left(\sum_{i\in I} e^{-\alpha|x-\rho\tilde\eta_i|}\right)^{-1}|E_2(x)|\lesssim e^{\alpha|x-\rho\tilde\eta_\ell|}\left|\(\sum\limits_{h=1}^k U_h\)^3-\sum\limits_{h=1}^k U_h^3\right|}.$$
Now by   Lemma \ref{chiave}
$${\begin{aligned}&\(\(\sum\limits_{h=1}^k U_h\)^3-\sum\limits_{h=1}^k U_h^3\)\\ &=\(U_1^3+3U_1^2\sum_{h\ge 2}^k U_h +3 U_1\(\sum_{h\ge 2}^k U_h\)^2+\(\sum_{h\ge 2}^k U_h\)^3-U_1^3-\sum_{h\ge 2}^k U_h^3\)\\
&\lesssim  U_1^2 \sum_{h\ge 2}^k U_h \\
&\lesssim e^{-\alpha|x-\rho\tilde\eta_\ell|}\sum_{h\ge 2}^k {e^{\alpha|x-\rho\tilde\eta_\ell|-2|x-\rho\xi_1|-|x-\rho\xi_h|}}\\
 &\lesssim  e^{-\alpha|x-\rho\tilde\eta_\ell|}e^{-\frac{2\pi\rho}{k}} 
\end{aligned}}$$
{because  (remind that $\xi_1=\tilde\eta_0$)  if $x\in  \widetilde\Sigma_0$   $$\begin{aligned}&-\alpha|x-\rho\tilde\eta_\ell|+2|x-\rho\xi_1|+|x-\rho\xi_h|\\ &=(2-\alpha)|x-\rho\xi_1|+|x-\rho\xi_h|\\ &\geq \rho|\xi_1-\xi_h|+(1-\alpha)|x-\rho\xi_1|\ge \rho|\xi_1-\xi_h|\end{aligned}$$ while if $x\in \widetilde\Sigma_\ell$ with $\ell\neq 0$ then $|x-\rho\tilde\eta_0|\ge|x-\rho\tilde\eta_\ell|$ and
$$\begin{aligned}&-\alpha|x-\rho\tilde\eta_\ell|+2|x-\rho\xi_1|+|x-\rho\xi_h|\\ &\ge \rho|\xi_1-\xi_h| +|x-\rho\xi_1|-\alpha|x-\rho\tilde\eta_\ell|\\ &= \rho|\xi_1-\xi_h|+(1-\alpha)|x-\rho\xi_1|+\alpha\(|x-\rho\tilde\eta_0|-|x-\rho\tilde\eta_\ell|\)\\&\ge \rho|\xi_1-\xi_h|.\end{aligned}$$ Then by \eqref{min1} $$\sum_{h= 2}^k e^{\alpha|x-\rho\tilde\eta_\ell|-2|x-\rho\xi_1|-|x-\rho\xi_h|}\lesssim \sum_{h=2}^k e^{-\rho|\xi_1-\xi_h|}\lesssim e^{-\rho\frac{2\pi}{k}}.$$  }

Therefore, the estimate
$$\|E_3\|_{*}\lesssim    e^{-\frac{2\pi\rho}{k}} $$ follows.

\item {\em Estimate of $E_3.$} {By Lemma \ref{gamma} and the fact that $\int\limits_{\mathbb R^n}(\partial_\rho W_\rho)^2\sim ck$ for some positive constant $c$,
we get
$$|\gamma_\rho|\lesssim  \frac{k }{  \rho^2 }e^{-4\rho\frac{\pi}{d k}}\ln\ln k \lesssim e^{-(1-\alpha)\frac{4\pi}{d k}\rho}.$$}

\end{itemize}
\end{proof}

\subsection{Solving the  intermediate non-linear problem (\ref{pbint}).}\label{6}
 
The next step relies on a classical contraction mapping argument.

\begin{proposition}\label{resto}
{For any compact set $B\subset \mathfrak B$  (see Remark \ref{wep}), there exists $k_0>0$ such that for any $\beta\in B$, for any $k\ge k_0$} and for any $\rho\in\mathcal D_k$, there is a unique  $ (\Phi,\mathfrak c)\in\mathscr H \times \mathbb R$  which  solves \eqref{pbint}. Moreover
   \beq\label{size}
   \|\Phi \|_{*}\lesssim\( \frac{1}{\rho^\nu}+{e^{-(1-\alpha)\frac{4\pi\rho}{dk}}}\).\eeq
\end{proposition}

\begin{proof} 
For a given $R>0$, let us consider the ball $$\mathcal B_k:=\left\{\Phi\in L^{\infty}(\mathbb R^n)\,:\, \|\Phi\|_{*}\le R\(\frac{1}{\rho^\nu}+{e^{-(1-\alpha)\frac{4\pi\rho}{dk}}}\)\right\}$$ which is a non-empty closed subset of $\mathscr B_{*}$.
Let us also introduce the map $\mathcal T:\mathcal B_k\cap\mathscr H \to \mathcal B_k \cap \mathscr H$ as 
$$\mathcal T(\Phi):=-\mathcal L^{-1}\(\mathcal N(\Phi)+\mathcal E\).$$
Now solving \eqref{pbint} is equivalent to find a fixed point to $\mathcal T$. \\ It is quite standard to prove that $\mathcal T$ is a contraction mapping for some $R$ provided $k$ is large enough.
Indeed, by Proposition \ref{inv}  
$$\|\mathcal T(\Phi)\|_{*}\lesssim \(\|\mathcal N(\Phi)\|_{*}+\|\mathcal E\|_{*}\)\ \hbox{and}\ \|\mathcal T(\Phi_1)-\mathcal T(\Phi_2)\|_{*} \lesssim \|\mathcal N(\Phi_1)-\mathcal N(\Phi_2)\|_{*}.$$
Moreover, by \eqref{qua} and since $W_\rho \lesssim 1$
 $$ 
\|\mathcal N(\Phi)\|_{*} \lesssim\|\Psi_\rho\|_{*}\|\Phi\|_{*}+\|\Phi\|_{*}^2+\|\Phi\|_{*}^3$$
and
$$ 
\|\mathcal N(\Phi_1)-\mathcal N(\Phi_2)\|_{*} \lesssim\(\|\Phi_1\|_{*}+\|\Phi_2\|_{*}+  \|\Psi_\rho\|_{*}\)\|\Phi_1-\Phi_2\|_{*}.$$
Finally, by Lemma \ref{errore} and \eqref{refin4} the claim follows.

\end{proof}

\subsection{The reduced problem} \label{7}
Finally, the problem reduces to find $\rho$ such that 
\begin{equation}\label{finale}\mathfrak C_k(\rho):=\intR\(\mathcal E +\mathcal N(\Phi)-\mathcal L(\Phi)\)\partial_\rho W_\rho=0,\end{equation}
 where  $\Phi=\Phi(\rho,k)$ and $\mathfrak c=\mathfrak c(\rho,k)$ are the solutions of \eqref{pbint} found in  Proposition \ref{resto}.\\

\begin{lemma}
\label{crho}{For any compact set $B\subset \mathfrak B$  (see Remark \ref{wep}), there exists $k_0>0$ such that for any $\beta\in B$, for any $k\ge k_0$} 
 and $\rho\in \mathcal D_k$  
$$\begin{aligned}\mathfrak C_k(\rho)&=\left\{\begin{aligned}& \mathfrak v_\infty\mathtt A\frac{k}{\rho^{\nu+1}}(1+o(1)) -\mathtt B \frac k\rho e^{-2\rho\frac\pi k}   (1+o(1))  - \beta   \mathtt C\(\frac k\rho\)^2  e^{-4\rho\frac{\pi}{d k}}  \ln\ln k (1+o(1))\\
&{+\beta^2 \Upsilon(k,\rho)}\ \hbox{if}\ n=3\\
& \mathfrak v_\infty\mathtt A  \frac{k}{\rho^{\nu+1}}(1+o(1)) -\mathtt B \sqrt{\frac k\rho}  e^{-2\rho\frac\pi k}   (1+o(1))  - \beta   \mathtt C\sqrt{\frac k\rho}  e^{-4\rho\frac{\pi}{d k}}  (1+o(1))\\
&{+\beta^2 \Upsilon(k,\rho)}\ \hbox{if}\ n=2\end{aligned}\right.\end{aligned}$$
where   $\mathtt A,$ $\mathtt B$ and $\mathtt C$  are positive constants only depending on $n$
and
\begin{equation}\label{finale00} |\Upsilon(k,\rho)|\lesssim k  {e^{-(1-\alpha) \frac{4\pi\rho}{d k}}}.\end{equation}

\end{lemma}
\begin{proof}
We know that $$\partial_\rho W_\rho =\sum_{h=1}^k \partial_\rho U_h =\sum_{h=1}^k \langle\nabla U_h, (-\xi_h) \rangle.$$

First, let us estimate the leading term
$$\begin{aligned}
\intR\mathcal E\partial_\rho W_\rho:=&\underbrace{\intR(1-V(x)) W_\rho\partial_\rho W_\rho}_{=I_1}+\underbrace{\intR\(\Delta W_\rho-W_\rho+W_\rho^3\)\partial_\rho W_\rho}_{=I_2}  \\
&+\underbrace{\gamma_\rho\intR\(\partial_\rho W_\rho\)^2}_{=I_3}\\
&+\underbrace{\intR 3 W_\rho\Psi_\rho^2\partial_\rho W_\rho+\intR\beta W_\rho\sum_{i=2}^d     \Psi^2_\rho(\hat\Theta_i x)\partial_\rho W_\rho}_{=:\Upsilon(k,\rho)=I_4}
\\ &+\underbrace{\intR2\beta \Psi_\rho\sum_{i=2}^d   W_\rho(\hat\Theta_i x) \Psi_\rho(\hat\Theta_i x)\partial_\rho W_\rho}_{=I_5}\\ &+\underbrace{\intR\(\Psi_\rho^3+ \beta \Psi_\rho\sum_{i=2}^d  \Psi^2_\rho(\hat\Theta_i x) \)\partial_\rho W_\rho}_{=I_6}
\end{aligned}$$

\begin{itemize}
\item{\em Estimate of $I_1$.}
 
By Lemma \ref{chiave}
$$\begin{aligned}I_1 &= k\int_{ \Sigma}(1-V(|x|))\(U_1+\sum_{h=2}^kU_h\) \(\partial_\rho U_1+\sum\limits_{i=2}^k \partial_\rho U_i\)\\ &=k\int_{ \Sigma}(1-V(|x|))U_1\partial_\rho U_1+h.o.t.
 \\
 &=k\int_{ \Sigma}(1-V(|x|))U(x-\rho\xi_1) U'(x-\rho\xi_1 ) {\langle x-\rho\xi_1,-\xi_1\rangle\over|x-\rho\xi_1|}dx+h.o.t.\\
 &=k\int_{ \mathbb R^3}(1-V(|y+\rho\xi_1|))U(y) U'(y) \(-{y_1\over|y|}\)dy+h.o.t.\\
 &=-k\int_{ \mathbb R^3}(1-V(|y+\rho\xi_1|))\frac 12\partial_{y_1}U^2(y)dy+h.o.t.\\
&=-\frac k 2 \intR  \partial_{y_1}V(|y+\rho\xi_1 |) U^2(y)dy+h.o.t.\\
&=\frac{k}{2}\frac{\nu  }{\rho^{\nu+1}}{\mathfrak v}_{\infty}  \intR U^2(y)dy +h.o.t.\\
\end{aligned}$$
because by \eqref{ipo-v} $$\partial_{x_1}V(|x|)=-\mathfrak v_\infty \nu\frac1{|x|^{\nu+1}}{x_1\over |x|}+\mathcal O\(\frac{1}{|x|^{\nu+1+\epsilon}}\).$$

 \item{\em Estimate of $I_2$.}
 
 By Lemma \ref{ruiz}, Lemma \ref{new} and Lemma \ref{chiave}
$$\begin{aligned} I_2&=k\int_{ \Sigma}\(\(U_1+\sum_{h=2}^kU_h\)^3-U_1^3-\sum_{h=2}^k U_h^3\)\(\partial_\rho U_1+\sum\limits_{i=2}^k \partial_\rho U_i\)\\
&= 3k\sum_{h=2}^k\int_{ \Sigma} U_1^2    U_h\partial_\rho U_1+h.o.t.\\ &=3 k \sum_{h=2} ^k\int_{ \Sigma} U^2(x-\rho\xi_1) U'(x-\rho\xi_1 ) {\langle x-\rho\xi_1,-\xi_1\rangle\over|x-\rho\xi_1|}   U (x-\rho\xi_h)dx+h.o.t.\\
&=  k \sum_{h=2} ^k  \intR \underbrace{3U^2(x+\rho(\xi_h- \xi_1)) U'(x+\rho(\xi_h- \xi_1)  ) {\langle x+\rho(\xi_h- \xi_1),-\xi_1\rangle\over|x+\rho(\xi_h- \xi_1) |} }_{=-\langle \nabla_\zeta \Gamma_{31}(\zeta),\xi_1\rangle,\ \zeta=\rho\(\xi_h-\xi_1 \)}
 U (x)dx+h.o.t.\\
&=   - \mathfrak c k \sum_{h=2}^k  \left\langle{\xi_1 -\xi_h \over|\xi_1 -\xi_h |},\xi_1\right\rangle 
{e^{-\rho|\xi_1 -\xi_h |}\over\(\rho|\xi_1 -\xi_h |\)^2 }+h.o.t.\\
&= - \frac1 2\mathfrak c  k{e^{-2\rho\frac\pi k}\over \rho} +h.o.t.\end{aligned}$$ 
because of \eqref{min1} and the fact that
\beq\label{fino1}|\xi_2-\xi_1|\sim {2\pi\over k}d\ \hbox{and}\ \langle\xi_1 -\xi_2,\xi_1\rangle=1-\cos{2\pi\over k} \sim  \frac12 \(2\pi\over k\)^2.\eeq

\item{\em Estimate of $I_3$.}
 The   term $I_3$ is estimated in Lemma \ref{gamma}.
\item{\em Estimate of $I_4$.}
We know that $\Psi_\rho=\beta Y_\rho$ where $Y_\rho$ solves \eqref{refin1} 
$$\begin{aligned}I_4&=\beta^2\(\intR 3 W_\rho Y_\rho^2\partial_\rho W_\rho+\intR\beta W_\rho\sum_{i=2}^d     Y^2_\rho(\hat\Theta_i x)\partial_\rho W_\rho\)\lesssim \beta^2 k{e^{-(1-\alpha)\frac{4\pi\rho}{dk}}}\end{aligned}$$
because  $\|Y_\rho\|_*\lesssim {e^{-(1-\alpha)\frac{4\pi\rho}{dk}}}$  (see \eqref{refin4}).
  
\item{\em Estimate of $I_5$.}
We observe that
$$|\Psi _\rho(\hat\Theta_i x)|\le \(\sum\limits_{\ell=1}^d\sum\limits_{j=1}^ke^{-\alpha|\hat\Theta_i x-\rho\eta_{\ell j}|}\)\|\Psi_\rho\|_{*}=\(\sum\limits_{\ell=1}^d\sum\limits_{j=1}^ke^{-\alpha|  x-\rho\eta_{\ell j}|}\)\|\Psi_\rho\|_{*}.$$
Therefore, since $|\partial_\rho W_\rho|\lesssim W_\rho$ by \eqref{inter2}
$$\begin{aligned}I_5&\lesssim \intR2  |\Psi_\rho| W_\rho \sum_{i=2}^d   W_\rho(\hat\Theta_i x) |\Psi_\rho(\hat\Theta_i x) |\\ &
\lesssim k\|\Psi_\rho\| ^2_{*}  \underbrace{\left\|W_\rho \sum_{i=2}^d   W_\rho(\hat\Theta_i x) \(\sum\limits_{\ell=1}^d\sum\limits_{j=1}^ke^{-\alpha|x-\rho\eta_{\ell j}|}\)\right\|_{*} }_{=\left\|W_\rho \sum_{i=2}^d   W_\rho(\hat\Theta_i x) |\right\|_{L^\infty(\mathbb R^n)}}\\
&\lesssim k{e^{-(1-\alpha)\frac{4\pi\rho}{dk}}}e^{-(1-\alpha)\frac{2\pi\rho}{kd}}\lesssim  k{e^{-(1-\alpha)\frac{6\pi\rho}{dk}}}.\end{aligned} $$
\item{\em Estimate of $I_6$.}
Using again that $|\partial_\rho W_\rho|\lesssim 1$, by \eqref{refin4}
$$I_6\lesssim k\|\Psi_\rho\|^3_{*}\lesssim k{e^{-(1-\alpha)\frac{6\pi\rho}{dk}}}.$$
 \end{itemize}

Next, by \eqref{qua}  
$$\intR \mathcal N(\Phi)\partial_\rho W_\rho\lesssim k \( \|\Phi\|^3_{*}+\|\Phi\|^2_{*}+\|\Psi_\rho\| _{*}\|\Phi\| _{*}\)\lesssim  k{e^{-(1-\alpha)\frac{6\pi\rho}{dk}}}.$$  
Finally, it only remains to estimate (see \eqref{lin})
$$\begin{aligned}\intR \mathcal L(\Phi)\partial_\rho W_\rho&=\underbrace{\intR\(V(x)-1\)\Phi \partial_\rho W_\rho}_{:=L_1}+\underbrace{\intR\(-\Delta\Phi+\Phi-3W_\rho^2\Phi \)\partial_\rho W_\rho}_{:=L_2}\\
&\underbrace{-\intR \beta \Phi\sum_{i=2}^d W^2_\rho(\hat\Theta_i x) \partial_\rho W_\rho-2\intR \beta W_\rho\sum_{i=2}^d W _\rho(\hat\Theta_i x)\Phi (\hat\Theta_i x) \partial_\rho W_\rho}_{:=L_3}.\end{aligned}$$
\begin{itemize}
\item{\em Estimate of $L_1$.}   Since $|\partial_\rho W_\rho|\lesssim  W_\rho$ {and since $\nu>1$}
$$\intR\( V(x)-1\)\Phi \partial_\rho W_\rho\lesssim k\|\( V(x)-1\)  W_\rho\|_{*}\|\Phi \|_{*}\lesssim k\frac 1{\rho^\nu}\|\Phi\|_{*}.$$
\item{\em Estimate of $L_2$.} Remind that $$\partial_\rho W_\rho =\sum_{h=1}^k \partial_\rho U_h =\sum_{h=1}^k \langle\nabla U_h, (-\xi_h) \rangle$$
 and also
$$-\Delta\partial _i U_h+\partial _iU_h=  \partial _iU_h^3\quad \Rightarrow\quad -\Delta\partial_\rho W_\rho+\partial_\rho W_\rho=\sum_{h=1}^k3U_h^2\langle\nabla U_h, (-\xi_h) \rangle $$

and so $$\begin{aligned}L_2&= \intR\(\sum_{h=1}^k3U_h^2\langle\nabla U_h, (-\xi_h) \rangle- 3 W^2_\rho\partial_\rho W_\rho \)\Phi\\ &\lesssim k\|\Phi\|_{*}\left\|\sum_{h=1}^k3U_h^2\langle\nabla U_h, (-\xi_h) \rangle- 3 W^2_\rho\partial _\rho W_\rho \right\|_{*}\\
&\lesssim k\|\Phi\|_{*}{e^{-\frac{2\pi\rho}k}}\lesssim  k{e^{-(1-\alpha)\frac{6\pi\rho}{dk}}}.\end{aligned}
$$
Indeed, if $x\in \Sigma$ then there is $\ell\in I$ such that $x\in \widetilde\Sigma_\ell$ (taking into account that $|\langle\nabla U_h, (-\xi_h) \rangle|\lesssim U_h$) arguing exactly as in the estimate of the term $E_2$ in Lemma \ref{errore} we get
$$\begin{aligned}&{e^{\alpha|x-\rho\tilde\eta_\ell|}\(\sum_{h=1}^k3U_h^2\langle\nabla U_h, (-\xi_h) \rangle- 3 W^2_\rho\partial _\rho W_\rho\)}\\ &
={e^{\alpha|x-\rho\tilde\eta_\ell|}\left[3U_1^2\langle\nabla U_1, (-\xi_1) \rangle+\sum_{h=2}^k3U_h^2(x)\langle\nabla U_h(x), (-\xi_h) \rangle\right.}\\ &\quad{\left.- 3 \(U_1+\sum_{h=2}^kU_h\)^2\( \langle\nabla U_1, (-\xi_1)+\sum_{h=2}^k \langle\nabla U_h, (-\xi_h)\rangle\)  \right]}\\
&{\lesssim e^{\alpha|x-\rho\tilde\eta_\ell|}\(U_1^2\sum_{h=2}^kU_h+U_1 \(\sum_{h=2}^kU_h\)^2\)\lesssim e^{-\frac{2\pi\rho}k}.}
\end{aligned}
$$
\item{\em Estimate of $L_3$.}  Since $\partial_\rho W_\rho\lesssim  W_\rho$ by Lemma \ref{inter}
$$  L_3\lesssim 
k\|\Phi\| _{*}\(\| W_\rho\sum_{i=2}^d W^2_\rho(\hat\Theta_i x)\|_{*}+ \| W_\rho\sum_{i=2}^d W_\rho(\hat\Theta_i x)\|_{*}\)\lesssim  k{e^{-(1-\alpha)\frac{6\pi\rho}{dk}}}.$$
\end{itemize}

 We combine all the previous estimates with the size of  $\Phi$  \eqref{size} and the size of $\Psi_\rho$ in \eqref{refin4}  and we get the claim.

\end{proof}

\subsection{The proof of Theorem \ref{main}: completed}\label{9}
We  have to find $\rho=\rho(k)$ such that $\mathfrak C_k(\rho )=0$ (see \eqref{finale}). 
We only consider the case $n=3.$ When $n=2$ we argue in a similar way.
Now,  by Lemma \ref{crho}
$$\begin{aligned}\mathfrak C_k(\rho)&=\mathfrak v_\infty\mathtt A  \frac{k}{\rho^{\nu+1}}(1+o(1)) -    \mathtt B \frac k\rho  e^{-2\rho\frac{\pi}{ k}}     (1+o(1)) - \beta   \mathtt C\(\frac k\rho\)^2  e^{-4\rho\frac{\pi}{d k}}  \ln\ln k (1+o(1))\\ &+\beta^2 \Upsilon (k,\rho) .
\end{aligned}$$
If $d=3$ and $ \nu >2$ or $d\ge4$ and $  \nu >1$ we  choose 
\beq\label{betak0}|\beta|\lesssim \beta_k:=\frac1{k^b}\ \hbox{with}\ 1<b<\nu {d-2\over2}\eeq
such that 
\beq\label{betak}\beta^2\Upsilon(k,\rho),\ \frac k\rho  e^{-2\rho\frac{\pi}{ k}}    =o\(\beta\(\frac k\rho\)^2  e^{-4\rho\frac{\pi}{d k}}  \ln\ln k \)\
\eeq
and so $\mathfrak C_k(\rho )$ reduces to
$$\begin{aligned}\mathfrak C_k(\rho)&=\mathfrak v_\infty\mathtt A  \frac{k}{\rho^{\nu+1}}(1+o(1))   - \beta   \mathtt C\(\frac k\rho\)^2  e^{-4\rho\frac{\pi}{d k}}  \ln\ln k (1+o(1)).
\end{aligned}$$
Therefore,   if $\rho=rk\ln k$  we compute  
$$\mathfrak C_k(rk\ln k)=\mathfrak v_\infty\mathtt A  \frac{1}{r^{\nu+1}(\ln k)^{\nu+1}}\frac1{k^\nu}(1+o(1))-  \beta\mathtt C \frac {\ln\ln k}{r^2(\ln k)^2} \frac 1{k^{{4\pi\over d} r}} (1+o(1))$$
and
$$\mathfrak C_k(rk\ln k)<0 \ \hbox{if}\ r<{d\nu\over 4\pi}\ \hbox{and}\
\mathfrak C_k(rk\ln k)>0 \ \hbox{if}\ r>{d\nu\over 4\pi}.$$
Therefore, if $\mathfrak v_\infty$ and $\beta$ have the same sign, for any $\epsilon>0$ there exists $ r(k)\in\({d\nu\over 4\pi}-\epsilon,{d\nu\over 4\pi}+\epsilon\),$  such that  $\mathfrak C_k(r(k)k\ln k)=0$ and the claim is proved.\\
It only remains to prove that \eqref{betak} holds for any $\rho=rk\ln k$ and $ r\in\({d\nu\over 4\pi}-\epsilon,{d\nu\over 4\pi}+\epsilon\)$ for some $\epsilon>0$ small enough.
We have
$$ {\beta^2\Upsilon(k,\rho)\over\beta\(\frac k\rho\)^2  e^{-4\rho\frac{\pi}{d k}}  \ln\ln k}\lesssim \beta k e^{\alpha\frac{4\pi\rho}{d k}}\lesssim 
k^{1-b+\alpha r\frac{4\pi}d}=o(1)\ \hbox{if}\ b>1$$
and
$$ {\frac k\rho  e^{-2\rho\frac{\pi}{ k}}  \over\beta\(\frac k\rho\)^2  e^{-4\rho\frac{\pi}{d k}}  \ln\ln k}\lesssim \frac1\beta  e^{ \frac{4\pi\rho}{d k}-\frac{2\pi\rho}k}\lesssim 
k^{b+  r\(\frac{4\pi}d-2\pi\)}=o(1)\ \hbox{if}\ b+\nu\(1-\frac d2\)<0.$$

Finally, the positivity of the solutions can be proved arguing as in \cite{PW} since $\beta$ is small. That completes the proof.

\section{Proof of Theorem \ref{main1}}\label{rela}
 
We consider the more general system
\begin{equation}\label{gen-sis}-\Delta u_i+V (x) u_i=|u_i|^{p-1}u_i+\beta |u_i|^{q-1}u_i\sum\limits_{j=1\atop j\not=i}^d |u_j|^{r-1}u_j\ \hbox{in}\ \mathbb R^n,\ i=1,\dots,d\end{equation}
with $p\in\(1,{n+2\over n-2}\),$ if $n\ge3$ or $p>1$ if $n=1,2$ and $q,r>1.$
Arguing as above we look for a solution to \eqref{gen-sis} as $\mf u=(u_1, \ldots, u_d)\in (H^1(\mathbb R^n))^d$ whose components satisfy
\eqref{comps} and    $u$ solves the non-local equation
\beq\label{gen-rid}
-\Delta u+V(x)u=|u|^{p-1}u+\beta |u|^{q-1}u \sum_{i=2}^d (|u|^{r-1}u)(\hat\Theta_i x), \quad\mbox{in}\,\, \mathbb R^n.\eeq
We build a  solution to \eqref{gen-rid} as 
$u=W_\rho +\Phi,$
where $W_\rho$ is defined in \eqref{ansatz}. In this case $U$  (see also \eqref{groundstate}) is the unique positive  radial solution to
$$-\Delta U+U=U^p\ \hbox{in}\ \mathbb R^n.$$
 The higher order term $\Phi\in \mathscr H $ satisfies the orthogonality condition 
\eqref{orto}.
We follow the same strategy developed in the previous sections.  Let us point out that we are not refining the ansatz as in the previous case.
This will create some constraints on the choice of the exponents $p,q,r$ and the number of equations $d$. We will focus on these constraints.\\

First, we write the non-local equation \eqref{gen-rid} in terms of $\Phi$, i.e.
$$\mathcal L(\Phi)=\mathcal N(\Phi)+\mathcal E\ \hbox{in}\ \mathbb R^n,$$
where 
the linear operator $\mathcal L$ is defined by
$$
\mathcal L(\Phi)=-\Delta\Phi+ \Phi-pW_\rho^{p-1}\Phi,$$
the error term $\mathcal E$ is defined by
\beq\label{err2}\begin{aligned}
\mathcal E:=&(1-V(x)) W_\rho +\Delta W_\rho-W_\rho+W_\rho^p+\beta W_\rho^q\sum_{i=2}^d  W_\rho^r(\hat\Theta_i x)  \\
\end{aligned}\eeq
and the higher order term $\mathcal N(\Phi)$ is defined by
$$\begin{aligned}
\mathcal N(\Phi)&:=|W_\rho+\Phi|^p-W_\rho^p-pW_\rho^{p-1}\Phi+\(V(x)-1\)\Phi\\
&+\beta\( |W_\rho+\Phi|^q\sum_{i=2}^d   |W_\rho+\Phi|^r(\hat\Theta_i x) -  W_\rho^q\sum_{i=2}^d  W_\rho^r(\hat\Theta_i x)\) .\end{aligned}$$
First of all, we invert the linear operator $\mathcal L$ in the Banach space  introduced in \cite[Section 3]{mpw}
 $$\mathscr B_{*}:=\{h\in L^\infty(\mathbb R^n)\,\,:\,\, \|h\|_{*}<+\infty\},\ \|h\|_{*}:=\sup_{x\in\mathbb R^n}\left(\sum_{j=1}^k e^{-\alpha|x-\rho\xi_j|}\right)^{-1}|h(x)|.$$ 
Here the $\xi_j$'s are defined in \eqref{points}
and  $\alpha\in(0, 1).$ We point out that in this case it is not necessary to refine the choice of the weight in the norm adding the points $\eta_{ij}$'s.
 \begin{proposition}\label{inv2}
There exist $k_0>0$   such that for any $k\ge k_0$,  $\rho\in\mathcal D_k$ and   $ h \in \mathscr B_{*}$ which satisfy \eqref{pari} and \eqref{rot}, the linear problem 
$$  \mathcal L(\Phi)= h +\mathfrak c \partial_\rho W_\rho\ \hbox{in}\  \mathbb R^n\ \hbox{and}\ \int\limits_{\mathbb R^n} \partial_\rho W_\rho \Phi=0$$
admits a unique solution $\Phi=\Phi_{\rho, k}\in\mathscr H\cap\mathscr B_{*}$ and $\mathfrak c=\mathfrak c(\rho,k)\in\mathbb R$ such that
$$
\|\Phi\|_{*}\lesssim \| h \|_{*}\quad\mbox{and}\quad |\mathfrak c|\lesssim\| h \|_{*}.$$
\end{proposition}
Next, we estimate  the error
\eqref{err2}.

\begin{lemma} 
It holds true that
$$\|\mathcal E\|_{*}\lesssim \frac{1}{\rho^\nu}+  e^{ -\min\{ 1,(p-1-\alpha)\}\frac{2\pi\rho}{k}} +e^{ -(\min\{   q,r\}-\alpha)\frac{2\pi\rho}{dk}}.$$

\end{lemma}
\begin{proof}
We know that
$$\begin{aligned} 
\|\mathcal E\|_{*}&\lesssim \underbrace{\|\(1-V(x)\)W_\rho  \|_{*}}_{:=E_1 } +\underbrace{\|\Delta W_\rho-W_\rho +W_\rho^p\|_{*}}_{:=E_2}+\underbrace{ \|W_\rho^q\sum_{i=2}^d W^r_\rho(\hat\Theta_i x)\|_{*}}_{:=E_3}.
\end{aligned}$$
$E_1$ and $E_2$ can be estimated as in Lemma \ref{errore}. Let us estimate the term $E_3.$
If $x\in \Sigma$ we have
$$\begin{aligned}  e^{\alpha|x-\rho\xi_1|} W_\rho^q(x)\sum_{\ell= 2}^dW^r_\rho(\hat\Theta_\ell x)&\lesssim 
e^{\alpha|x-\rho\xi_1|} U ^q(x-\rho\xi_1) \sum\limits_{\ell\in I\atop\ell\not=0} U^r(x-\rho\tilde \eta_\ell)\\ &
\lesssim
 \sum\limits_{\ell\in I\atop\ell\not=0}  e^{-(q-\alpha)|x-\rho\xi_1|-r|x-\rho\tilde\eta_\ell|}\\ &\lesssim e^{- (q-\alpha)\frac{2\pi\rho}{dk}},
\\ &
\end{aligned}$$
because  
$$\begin{aligned}(q-\alpha) |x-\rho\xi_1|+r|x-\rho\tilde \eta_\ell|&\ge (q-\alpha)\rho|\tilde\eta_0-\tilde\eta_\ell|+(r-q+\alpha)|x-\rho\tilde \eta_\ell|\\ &\ge (q-\alpha)\rho\sin{2\pi\over dk}\ \hbox{if}\ q\le r\end{aligned} $$
and
$$\begin{aligned}(q-\alpha) |x-\rho\xi_1|+r|x-\rho\tilde \eta_\ell|&\ge r\rho|\tilde\eta_0-\tilde\eta_\ell|+(q-\alpha-r)|x-\rho\tilde \eta_\ell|\\ &\ge r\rho\sin{2\pi\over dk}
\ \hbox{if}\ q> r.\end{aligned} $$
Then the estimate
$$\|E_3\|_{*} \lesssim \left\{\begin{aligned}&  e^{-(q-\alpha)\frac{2\pi\rho}{d k}}\ \hbox{if}\ q\le r\\ 
& e^{-r\frac{2\pi\rho}{d k}}\ \hbox{if}\ q>r\end{aligned}\right.$$
follows.

\end{proof}
Now, we use a standard contraction mapping argument to solve the  intermediate non-linear problem  \eqref{pbint}.

\begin{proposition}
  There exists $k_0$ such that for any $k\ge k_0$ and for any $\rho\in\mathcal D_k$, there is a unique  $ (\Phi,\mathfrak c)\in\mathscr H \times \mathbb R$  which  solves \eqref{pbint}. Moreover
\beq
   \|\Phi \|_{*}\lesssim   \frac{1}{\rho^\nu}+  e^{ -\frac{2\pi\rho}{k}\sigma} , \ \sigma:=\min\left\{ 1,(p-1-\alpha),\frac{q-\alpha}d,\frac{r-\alpha}d\right\}.\label{resto2}\eeq
\end{proposition}

 Finally, we have to find $\rho$ such that 
 $$\mathfrak C_k(\rho):=\int\limits_{\mathbb R^n}\(\mathcal E+\mathcal N(\Phi)-\mathcal L(\Phi)\)=0.$$
and in the next lemma we estimate $\mathfrak C_k(\rho).$
\begin{lemma}
\label{errore2}
It holds true that
\begin{enumerate}
\item if $q+1\not=r$ 
$$\begin{aligned}\mathfrak C_k(\rho)&=\mathfrak v_\infty\mathtt A_1\frac{k}{\rho^{\nu+1}}(1+o(1))-  \mathtt A_2 \( \frac k\rho \)^{n-1\over2}e^{-
\frac{2\pi\rho} k} (1+o(1))  \\
&-\beta   \mathtt A_3\(\frac k\rho\)^{\min\{q+1,r\}\frac{n-1}2}  e^{-\min\{q+1,r\}\frac{2\pi\rho}{d k}} (1+o(1))+\Xi_k(\rho) ,\end{aligned}$$
\item if $q+1=r>\frac{n+1}{n-1}$
$$\begin{aligned}\mathfrak C_k(\rho)&=\mathfrak v_\infty\mathtt A_1\frac{k}{\rho^{\nu+1}}(1+o(1))-  \mathtt A_2 \( \frac k\rho \)^{n-1\over2}e^{-
\frac{2\pi\rho} k} (1+o(1))  \\
&-\beta   \mathtt A_3\(\frac k\rho\)^{r\frac{n-1}2}  e^{-r\frac{2\pi\rho}{d k}} (1+o(1))+\Xi_k(\rho) ,\end{aligned}$$
\item if $q+1=r=\frac{n+1}{n-1}$
$$\begin{aligned}\mathfrak C_k(\rho)&=\mathfrak v_\infty\mathtt A_1\frac{k}{\rho^{\nu+1}}(1+o(1))-  \mathtt A_2 \( \frac k\rho \)^{n-1\over2}e^{-
\frac{2\pi\rho} k} (1+o(1))  \\
&-\beta   \mathtt A_3\(\frac k\rho\)^{r\frac{n-1}2} e^{-r\frac{2\pi\rho}{d k}}\ln\ln k  (1+o(1)) +\Xi_k(\rho) ,\end{aligned}$$
\item if $q+1=r<\frac{n+1}{n-1}$
$$\begin{aligned}\mathfrak C_k(\rho)&=\mathfrak v_\infty\mathtt A_1\frac{k}{\rho^{\nu+1}}(1+o(1))-  \mathtt A_2 \( \frac k\rho \)^{n-1\over2}e^{-
\frac{2\pi\rho} k} (1+o(1))   \\
&-\beta   \mathtt A_3\(\frac k\rho\)^{r(n-1)+\frac{n-1}2} e^{-r\frac{2\pi\rho}{d k}}   (1+o(1))+\Xi_k(\rho) ,\end{aligned}$$
\end{enumerate}
where the $\mathtt A_i$'s  are positive constants
$$\Xi_k(\rho)=k\(\|\Phi\|_{*}^2+\({1\over\rho^\nu}+e^{- \frac{2\pi\rho}{k}\tau}\) \|\Phi\|_{*}\),\ \tau:=\min\left\{1,p-1-\alpha,\frac{q-1-\alpha}d,\frac{r-1-\alpha}d\right\}$$
and $\Phi$ satisfies \eqref{resto2}.
\end{lemma}
\begin{proof}
We argue as in Lemma \ref{crho}.
The leading term is $\intR\mathcal E\partial_\rho W_\rho:$
$$\begin{aligned}
\intR\mathcal E\partial_\rho W_\rho:=&\underbrace{\intR(1-V(x)) W_\rho\partial_\rho W_\rho}_{=I_1}+\underbrace{\intR\(\Delta W_\rho-W_\rho+W_\rho^p\)\partial_\rho W_\rho}_{=I_2}  \\
&+\underbrace{\int\limits_{\mathbb R^n}W_\rho^q\sum_{i=2}^d W^r_\rho(\hat\Theta_i x)\partial_\rho W_\rho}_{=I_3}
\end{aligned}$$
The first term $I_1$ is given in Lemma \ref{crho}. 
We estimate $I_2$  (see the estimate of $I_2$ in  Lemma \ref{crho})
$$\begin{aligned} I_2&=k\int_{ \Sigma}\(\(U_1+\sum_{h=2}^kU_h\)^p-U_1^p-\sum_{h=2}^k U_h^p\)\(\partial_\rho U_1+\sum\limits_{i=2}^k \partial_\rho U_i\)\\
&= k\sum_{h=2}^k\int_{ \Sigma} pU_1^{p-1}    U_h\partial_\rho U_1+h.o.t.\\ &= k \sum_{h=2} ^k\int_{ \Sigma} pU^{p-1}(x-\rho\xi_1) U'(x-\rho\xi_1 ) {\langle x-\rho\xi_1,-\xi_1\rangle\over|x-\rho\xi_1|}   U (x-\rho\xi_h)dx+h.o.t.\\
&=  k \sum_{h=2} ^k  \underbrace{\intR pU^{p-1}(x+\rho(\xi_h- \xi_1)) U'(x+\rho(\tilde\eta_\ell- \xi_1)  ) {\langle x+\rho(\xi_h- \xi_1),-\xi_1\rangle\over|x+\rho(\xi_h- \xi_1) |} U (x)dx}_{= \langle \nabla_\zeta \Gamma_{p,1}(\zeta),\xi_1\rangle,\ \zeta=\rho\(\xi_h-\xi_1 \)}
 +h.o.t.\\
&=    - \mathfrak  c  k \sum_{h=2}^k {\left\langle   {\xi_1 -\xi_h\over |\xi_1 -\xi_h |},\xi_1\right\rangle } {e^{-\rho|\xi_1 -\xi_h |}\over\(\rho|\xi_h-\xi_1|\)^{n-1\over2}}+h.o.t.\\
&= -  \mathtt A_2  \(\frac k\rho\)^{n-1\over 2}e^{-2\rho\frac\pi k} +h.o.t. \ \hbox{because of \eqref{fino1}}\end{aligned}$$ 
and also $I_3$ (see the estimate of $I_3$ in  Lemma \ref{crho})
$$\begin{aligned}I_3= &\beta\intR\(\sum_{h=1}^k U_h\)^q\sum_{i=2}^d \(\sum_{j=1}^k U_j(\hat \Theta_i x)\)^r\partial_\rho W_\rho\\ &=\beta k\int_{ \Sigma}\(U_1+\sum_{h=2}^k U_h\)^q\sum_{i=2}^d\(\sum_{j=1}^k U _{ij}\)^r\(\partial_\rho U_1+\sum\limits_{i=2}^k \partial_\rho U_i\)\\
&=\beta k \int_{ \Sigma} U_1^q\partial_\rho U_1\sum_{i=2}^d  \sum_{j=1}^kU_{ij}^r+h.o.t. \\
&=\beta k \sum_{\ell\in \mathcal I\atop \ell\not=0}\int_{ \Sigma} U^q(x-\rho\xi_1) U'(x-\rho\xi_1 ) {\langle x-\rho\xi_1,-\xi_1\rangle\over|x-\rho\xi_1|}   U^r(x-\rho\tilde\eta_{\ell})dx+h.o.t.\\
&=\beta k \sum_{\ell\in \mathcal I\atop \ell\not=0}\underbrace{\int_{ \mathbb R^n} U^q(x+\rho(\tilde\eta_{\ell}- \xi_1)) U'(x+\rho(\tilde\eta_\ell- \xi_1)  ) {\langle x+\rho(\tilde\eta_\ell- \xi_1),-\xi_1\rangle\over|x+\rho(\tilde\eta_{\ell}- \xi_1) |}  U^r(x)dx}_{=- \frac1{q+1}\langle\nabla_\zeta \Gamma_{q+1,r}(\zeta),\xi_1\rangle,\ \zeta=\rho\(\tilde\eta_{\ell}-\xi_1 \)}
+h.o.t.\\
 &= \left\{\begin{aligned}
&-  \beta   \mathtt A_3\(\frac k\rho\)^{\min\{q+1,r\}\frac{n-1}2} e^{-\min\{q+1,r\} \frac{2\pi\rho}{d k}} +h.o.t. \  \hbox{if}\ r\not =q+1 \\
&-  \beta   \mathtt A_3\(\frac k\rho\)^{r(n-1)+\frac{n+1}2} e^{-r \frac{2\pi\rho}{d k}}   +h.o.t. \  \hbox{if}\ r=q+1<{n+1\over n-1}\\
&-  \beta   \mathtt A_3\(\frac k\rho\)^{r(n-1)\over2} e^{-r \frac{2\pi\rho}{d k}}  \ln\ln k +h.o.t. \  \hbox{if}\ r=q+1={n+1\over n-1}\\
&-  \beta   \mathtt A_3\(\frac k\rho\)^{r(n-1)\over2} e^{-r \frac{2\pi\rho}{d k}}   +h.o.t. \  \hbox{if}\ r=q+1>{n+1\over n-1}\\
\end{aligned}\right.\end{aligned}$$
because of  \eqref{fino2}.

Moreover, the higher order terms $\int\limits_{\mathbb R^n}\mathcal N(\Phi)\partial_\rho W_\rho$ and $\int\limits_{\mathbb R^n}\mathcal L(\Phi)\partial_\rho W_\rho$ can be estimated as
$$\int\limits_{\mathbb R^n}\mathcal N(\Phi)\partial_\rho W_\rho\lesssim k\(\|\Phi\|_{*}^2+\left\|W_\rho^{q-1}\sum\limits_{\ell=2}^d|W_\rho(\hat\Theta_i x)|^{r}\right\|_{*}\|\Phi\|_{*}+\left\|W_\rho^q\sum\limits_{\ell=2}^d|W_\rho(\hat\Theta_i x)|^{r-1}\right\|_{*}\|\Phi\|_{*}\)$$
where (arguing as in Lemma \ref{inter})
$$\left\|W_\rho^{q-1}\sum\limits_{\ell=2}^d|W_\rho(\hat\Theta_i x)|^{r}\right\|_{*}\lesssim e^{-\(\min\{q-1,r\}-\alpha\)\frac{2\pi\rho}{dk}}$$
and
$$\left\|W_\rho^{q}\sum\limits_{\ell=2}^d|W_\rho(\hat\Theta_i x)|^{r-1}\right\|_{*}\lesssim e^{-\(\min\{q,r-1\}-\alpha\)\frac{2\pi\rho}{dk}},$$
and 
 $$\int\limits_{\mathbb R^n}\mathcal L(\Phi)\partial_\rho W_\rho\lesssim k\(\frac1{\rho^\nu}\|\Phi\|_{*}+e^{-\min\{1,(p-1-\alpha)\}\frac{2\pi\rho}{k}}\|\Phi\|_{*}\),$$
since we have (see the estimate of $L_2$ in Lemma \ref{crho})
$$\begin{aligned}& \sum_{h=1}^kU_h^{p-1}\langle\nabla U_h, (-\xi_h) \rangle-  W^{p-1}_\rho\partial _\rho W_\rho  \\ &
=  U_1^{p-1}\langle\nabla U_1, (-\xi_1) \rangle+\sum_{h=2}^k U_h^{p-1} \langle\nabla U_h , (-\xi_h) \rangle\\ &-  \(U_1+\sum_{h=2}^kU_h\)^{p-1}\( \langle\nabla U_1, (-\xi_1)+\sum_{h=2}^k \langle\nabla U_h, (-\xi_h)\rangle\)  \\ &
\lesssim U_1^{p-1} \sum_{h=2}^k U_h +U_1 \sum_{h=2}^k U_h^{p-1}.\end{aligned}$$
\end{proof}

 At this point it is clear that  a solution to the non-local equation \eqref{gen-rid} does exist if 
we find $\rho$ so that $\mathfrak C_k(\rho)=0.$ At this aim, it is useful that  the last term $\Xi_k(\rho)$ is an higher order term in the expansion of $\mathfrak C_k$ and this is achieved if we choose the exponents $p,$ $q$ and $r$ and the number of equations $d$ in a proper way. 
We will focus on the particular case  $q=\frac{p-1}2$ and $r=\frac{p+1}2.$ 

\begin{proof}[ Proof of Theorem \ref{main1}: completed]
If   $d>\frac{p+1}2$ and $p>5$ ($n=2$ because the exponent $5$ is critical in 3D), by (2) of Lemma \ref{errore2}
  \beq\label{ckr} \mathfrak C_k(\rho) =\mathfrak v_\infty\mathtt A_1\frac{k}{\rho^{\nu+1}}(1+o(1))
-\beta   \mathtt A_3 \(\frac k\rho\)^{\frac{p+1}{4}} e^{-\frac{(p+1)\pi\rho}{d k}}(1+o(1))  .\eeq
 Indeed, in this case $\frac{p+1}{2d}<1$ and  $\Xi_k(\rho)$ can be estimated using \eqref{resto2} with
  $$\sigma=\frac {p-1}{2 d},\ \tau=\frac{p-3}{2d},\  2\sigma>\frac{p+1}{2d}\ \hbox{and}\ \sigma+\tau>\frac{p+1}{2d}.$$
   By \eqref{ckr}, there exists    $\rho(k)\in \mathcal D_k$ (see \eqref{raggio})  such that  $ \mathfrak C_k(\rho(k)) =0$ if either  $ \mathfrak v_\infty$ and $\beta$ have the same sign.
   
  On the other hand if  $d\le \frac{p-1}2$ and $p>3 $  the coupling term is an higher order term in the expansion (2), (3) or (4) of Lemma \ref{errore2} and 
   \beq\label{ckr2} \mathfrak C_k(\rho) =\mathfrak v_\infty\mathtt A_1\frac{k}{\rho^{\nu+1}}(1+o(1))
-\beta   \mathtt A_2 \(\frac k\rho\)^{\frac{n-1}2} e^{-\frac{2\pi\rho}{  k}}(1+o(1))  .\eeq
 Indeed in this case $ 1<\frac{p+1}{2d}$ and
    $\Xi_k(\rho)$ can be estimated using \eqref{resto2} with
   $$\sigma=1,\ \tau>0,\  2\sigma>1\ \hbox{and}\ \sigma+\tau>1.$$
     By \eqref{ckr2}, there    $\rho(k)\in \mathcal D_k$ (see \eqref{raggio})  such that  $ \mathfrak C_k(\rho(k)) =0$ if  $ \mathfrak v_\infty>0$ whatever $\beta$ is.
   
\end{proof}

 \appendix\section{Auxiliary results}\label{app1}
 
We  recall the  result  \cite[Lemma 3.7]{ACR}.
\begin{lemma}\label{ruiz}
Let $W_1,W_2:\mathbb R^n\to \R$ be two positive continuous radial functions such that $$W_i(x)\sim |x|^{-a_i} e^{-b_i|x|} \ \hbox{as}\ |x|\to\infty$$ where $a_i\in\R$, $b_i>0$. Then for some constant $\mathfrak c>0$ we have
\begin{itemize}
\item[(i)] If $b_1<b_2$ then $$\intR W_1(x+\zeta) W_2(x)dx\sim \mathfrak c e^{-b_1|\zeta|}|\zeta|^{a_1}\ \hbox{as}\ |\zeta|\to\infty$$
Clearly, if $b_1>b_2$ a similar expression holds, by replacing $a_1$ and $b_1$ with $a_2$ and $b_2$.
\item[(ii)] If $b_1=b_2=:b$ then, suppose that $a_1\le a_2$  $$\intR W_1(x+\zeta) W_2(x)dx\sim\left\{\begin{aligned} &\mathfrak c  e^{-b|\zeta|}|\zeta|^{-a_1-a_2+\frac{n+1}{2}}\quad &\mbox{if}\ a_2<\frac{n+1}{2}\\
&\mathfrak c e^{-b|\zeta|}|\zeta|^{-a_1}\ln|\zeta|\quad &\mbox{if}\ a_2=\frac{n+1}{2}\\
&\mathfrak c e^{-b|\zeta|}|\zeta|^{-a_1}\quad &\mbox{if}\ a_2>\frac{n+1}{2}\end{aligned}\right.\quad \hbox{as}\ |\zeta|\to\infty.$$\end{itemize}
\end{lemma}

 Let $s,t\ge1.$ Set
$$\Gamma_{s,t}(\zeta):=\intR U^s(x+\zeta) U^t(x)dx,\ \zeta\in\mathbb R^n.$$
By Lemma \eqref{ruiz} and \eqref{groundstate}
there exists $\mathfrak c>0$ such that
\begin{itemize}
\item[(i)] if $s<t$ then $$\Gamma_{s,t}(\zeta)\sim \mathfrak c e^{-s|\zeta|}|\zeta|^{-s{n-1\over2}}\ \hbox{as}\ |\zeta|\to\infty$$
 \item[(ii)] if $s=t$ then  
 $$\Gamma_{s,t}(\zeta)\sim\left\{\begin{aligned} &\mathfrak c  e^{-s|\zeta|}|\zeta|^{-s(n-1)-\frac{n+1}{2}}\quad &\mbox{if}\ s <\frac{n+1}{n-1}\\
&\mathfrak c e^{-s|\zeta|}|\zeta|^{-s{n-1\over2}}\ln|\zeta|\quad &\mbox{if}\ s= \frac{n+1}{n-1}\\
&\mathfrak c e^{-s|\zeta|}|\zeta|^{-s{n-1\over2}}\quad &\mbox{if}\ s >\frac{n+1}{n-1}\end{aligned}\right.\quad \hbox{as}\ |\zeta|\to\infty.$$\end{itemize}

 We are going to prove that all the previous estimates hold true in the $C^1-$sense. 
\begin{lemma}\label{new}
It holds true that
\begin{itemize}
\item[(i)] if $s<t$ then $$\nabla_\zeta \Gamma_{s,t}(\zeta)\sim -  \mathfrak c s\frac\zeta{|\zeta|} e^{-s |\zeta|}|\zeta|^{-s{n-1\over2}} \ \hbox{as}\ |\zeta|\to\infty$$
 \item[(ii)] if $s=t$ then, suppose that $a_1\ge a_2$  
 $$\nabla_\zeta \Gamma_{s,t}(\zeta)\sim \left\{\begin{aligned} &-  \mathfrak c s\frac\zeta{|\zeta|}  e^{-s|\zeta|}|\zeta|^{-s(n-1)-\frac{n+1}{2}}\quad &\mbox{if}\ s <\frac{n+1}{n-1}\\
&-  \mathfrak c s\frac\zeta{|\zeta|} e^{-s|\zeta|}|\zeta|^{-s{n-1\over2}}\ln|\zeta|\quad &\mbox{if}\ s= \frac{n+1}{n-1}\\
&-  \mathfrak c s\frac\zeta{|\zeta|} e^{-s|\zeta|}|\zeta|^{-s{n-1\over2}}\quad &\mbox{if}\ s >\frac{n+1}{n-1}\end{aligned}\right.\quad \hbox{as}\ |\zeta|\to\infty.$$\end{itemize}
\end{lemma}
\begin{proof}
We only prove the case $s<t,$ being the proof of other cases similar. We point out that
$$\int\limits_{\mathbb R^n}U^s(x+ \zeta)\partial_{x_i}U^t(x)dx=-\int\limits_{\mathbb R^n}\partial_{x_i}U^s(x+ \zeta)U^t(x)dx=-\int\limits_{\mathbb R^n}\partial_{\zeta_i}U^s(x+ \zeta)U^t(x)dx$$
and we are going to prove that
$$\int\limits_{\mathbb R^n}\partial_{\zeta_i}U^s(x+ \zeta)U^t(x)dx\sim-  \mathfrak c s\frac\zeta{|\zeta|} e^{-s |\zeta|}|\zeta|^{-s{n-1\over2}}\ \hbox{as}\ |\zeta|\to\infty.$$

Set $f(\zeta):=g(\zeta)h(\zeta) $ with
$$g(\zeta):=\int\limits_{\mathbb R^n}U^s(x+ \zeta)U^t(x)dx\ \hbox{and}\ h(\zeta):=e^{s |\zeta|}|\zeta|^{s{n-1\over2}}.$$ 
We know that
\begin{equation}\label{n0}\lim\limits_{|\zeta|\to \infty}f(\zeta)=\mathfrak c,\end{equation}
We are going to prove that
\begin{equation}\label{n1}
\lim\limits_{|\zeta|\to \infty}\partial_{\zeta_i} f(\zeta)=0.\end{equation}
Since
\begin{equation}\label{n2}\partial_{\zeta_i}f =g \partial_{\zeta_i}h +h \partial_{\zeta_i}g \end{equation}
and
\begin{equation}\label{n3}\partial_{\zeta_i}h(\zeta)=s{\zeta_i\over|\zeta|}h(\zeta)(1+o(1)),\end{equation}
by \eqref{n0}, \eqref{n1}, \eqref{n2} and \eqref{n3}   the claim follows.
 To prove \eqref{n1}, by Lemma \ref{calcolo1} we need to show that  there exists $c>0$ such that
$$
|\partial^2_{\zeta_i\zeta_i}f(\zeta)|\le c\ \hbox{if $|\zeta|$ is large enough}.$$
We have
$$\partial^2_{\zeta_i\zeta_i}f=g\partial^2_{\zeta_i\zeta_i}h+2\partial_{\zeta_i}g\partial_{\zeta_i}h +h\partial^2_{\zeta_i\zeta_i}g.$$
It is easy to check that
$$  |\partial_{\zeta_i}h |,|\partial^2_{\zeta_i\zeta_i}h |=\mathcal O\(h \).$$
By \eqref{n0} 
$g   =\mathcal O\(\frac1{h }\).$
We only need to estimate $\partial_{\zeta_i} g$ and $\partial^2_{\zeta_i\zeta_i}g$.
We have
$$\begin{aligned}\partial_{\zeta_i} g(\zeta)&=\int\limits_{\mathbb R^n}\partial_{\zeta_i} U^s(x+  \zeta)U^t(x)dx\\ & =\int\limits_{\mathbb R^s}sU^{s-1} (x+ \zeta) U'(x+\zeta){x_1+\zeta_i\over|x+\zeta|}  U^t(x)dx\\
&\lesssim  \int\limits_{\mathbb R^n}U^{s-1} (x+ \zeta) |U'(x+\zeta)|U^t(x)dx\ \hbox{\tiny since $|U'|=\mathcal O(U)$}
\\ &\lesssim \int\limits_{\mathbb R^n}U^ s (x+ \zeta)  U^t(x)dx\lesssim \frac1{h(\zeta)} \ \hbox{\tiny because of \eqref{n0}}
\end{aligned}$$
and
$$\begin{aligned}\partial^2_{\zeta_i\zeta_i}g(\zeta)&= \int\limits_{\mathbb R^n}s\partial_{\zeta_i}\(U^{s-1} (x+ \zeta) U'(x+\zeta){x_1+\zeta_i\over|x+\zeta|}\) U^t(x)dx\\
& = \int\limits_{\mathbb R^n}sU^{s-2}\( U'(x+\zeta){x_1+\zeta_i\over|x+\zeta|}\) ^2U^t(x)dx\\ &+
 \int\limits_{\mathbb R^n}sU^{s-1} (x+ \zeta) U''(x+\zeta)\({x_1+\zeta_i\over|x+\zeta|}\)^2 U^t(x)dx\\ &+
 \int\limits_{\mathbb R^n}sU^{s-1} (x+ \zeta) U'(x+\zeta) \({1\over |x+\zeta|}-{(x_1+\zeta_i)^2\over|x+\zeta|^3} \) U^t(x)dx\\
&\\
& \lesssim \int\limits_{\mathbb R^n}   U^s(x+\zeta) U^t(x)dx \ \hbox{\tiny since $|U'|,|U''|=\mathcal O(U)$}\\ &+
\underbrace{ \int\limits_{\mathbb R^n}{1\over |x+\zeta|}  U ^s(x+ \zeta) U^t(x)dx}_{=\mathcal O\(e^{-s|\zeta|} |\zeta|^{-s{n-1\over2}-1}\)} \ \hbox{\tiny because of Lemma \ref{ruiz} with $b_1=s<b_2=t$ and $a_1=-s{n-1\over2}-1$}\\ & 
\\ &\lesssim   \frac1{h(\zeta)} \ \hbox{\tiny because of \eqref{n0}}.
\end{aligned}$$
That concludes the proof.
\end{proof}

\begin{lemma}\label{calcolo1}
Let $f:\mathbb R^n\to\mathbb R$ be a $C^2-$function such that
$$\lim\limits_{|x|\to\infty} f(x)=l\in\mathbb R$$
and there exists $c>0$ such that
$$\sup\limits_{x\in\mathbb R^n}|\partial_{x_1x_1}^2f(x)|\le c$$
then
$$\lim\limits_{|x|\to\infty} \partial_{x_1} f(x)=0.$$
\end{lemma}
\begin{proof}
Let $\epsilon>0$ be fixed. If  $\delta=\frac{\epsilon^2}{4c}$ there exists $R>0$ such that if $|x|\ge R$ then $|f(x)-l|\le\delta.$
Now take $e_1:=(1,0,\dots,0)\in\mathbb R^n$, $t\in (-1,1)$ so that $|x+te_1|\ge R$ if $|x|\ge R+1$ and apply mean value theorem
$$f(x+te_1)=f(x)+\partial_{x_1}f(x) t+ \frac12\partial_{x_1x_1}^2f(x+\epsilon te_1)t^2$$ for some $\epsilon\in(0,1).$
Then if $|x|\ge R+1$
$$|\partial_{x_1}f(x)|\le \frac1 {|t|} 2\delta+\frac12 c|t|\ \hbox{for any}\ t\in (-1,1),\ t\not=0\ \Rightarrow\ |\partial_{x_1}f(x)|\le2 \sqrt{\delta c}=\epsilon$$
and the claim follows.
\end{proof}

\section{Proof of Proposition \ref{inv}}\label{app2}
Let us consider the Hilbert space
$W^{1,2}(\mathbb R^n)$ equipped with the scalar product
$$\langle u,v\rangle=\int\limits_{\mathbb R^n}\(\nabla u \nabla v+uv\)dx.$$
It is well known that the embedding $\mathscr I: W^{1,2}(\mathbb R^n)\hookrightarrow  L^{ 2}(\mathbb R^n)$ is continuous. We define the adjoint
operator $\mathscr I^*:    L^{ 2}(\mathbb R^n)\to W^{1,2}(\mathbb R^n)$ by duality, i.e.
$$\mathscr I^* f=u\ \hbox{if and only if $u$ is a weak solution to }\ -\Delta u+u=f\ \hbox{in}\ \mathbb R^n.$$
Let us introduce the space
 $$H_\rho^\perp:=\left\{\Phi\in \mathscr H\ :\ \langle \Phi,\mathscr I^* \(\partial_\rho W_\rho\)\rangle=\int\limits_{\mathbb R^n} \Phi\partial_\rho W_\rho=0\right\}$$
 where $\mathscr H$ is defined in \eqref{acca}.

 It is immediate to check that linear problem
\beq\label{pro} \mathcal L (\Phi)= h \ \hbox{in}\  \mathbb R^n ,\ \int\limits_{\mathbb R^n}\Phi\partial_\rho W_\rho=\int\limits_{\mathbb R^n}h\partial_\rho W_\rho=0  \end{equation}
can be rewritten as
$$\Phi+\mathscr K\(\Phi\)=h^*:=\mathscr I^* h,\quad h^*, \Phi\in H_\rho^\perp,$$
where
$$\mathscr K\(\Phi\)=\mathscr I^*\left[\(V(x)-1\)\Phi-3W_\rho^2\Phi  -\beta \Phi\sum_{i=2}^d W^2_\rho(\hat\Theta_i x) -2\beta W_\rho\sum_{i=2}^d W _\rho(\hat\Theta_i x)\Phi(\hat\Theta_i x)\right]$$
is a compact operator since $W_\rho$ decays exponentially and $V$ satisfies \eqref{ipo-v}.
Therefore, by  Fredholm alternative solving problem \eqref{pro} is equivalent to prove that it has a unique solution when $h=0.$ 
Now, it is important to point out that if $h\in \mathscr B_{*}$ 
 then the solution $\Phi\in W^{1,2}(\mathbb R^n)$ to \eqref{pro} also belongs to $ \mathscr B_{*}$, i.e. it belongs to $L^\infty(\mathbb R^n).$
It is enough to apply the standard regularity theory. Indeed,  by \eqref{normastella} we deduce that $h\in L^q(\mathbb R^n)$ for any $q>1$. Moreover, since $W_\rho\in L^\infty(\mathbb R^n)$ and   $\Phi\in L^2(\mathbb R^n),$ by \eqref{pro} we immediately deduce that $\Phi\in W^{2,2}(\mathbb R^n),$ which is embedded in $L^{\infty}(\mathbb R^n)$ if $n=2,3.$

 Finally, to prove the Proposition \ref{inv} it is enough to prove the a priori estimate \eqref{norma}.
 
At this aim it is useful to 
 decompose $\mathcal L$ as  
$$\mathcal L(\Phi)=\underbrace{-\Delta\Phi+V(x)\Phi-3W_\rho^2\Phi  -\beta \Phi\sum_{i=2}^d W^2_\rho(\hat\Theta_i x)}_{\mathcal L_0 (\Phi)}\underbrace{-2\beta W_\rho\sum_{i=2}^d W _\rho(\hat\Theta_i x)\Phi(\hat\Theta_i x)}_{\mathcal L_1(\Phi)}.$$
and to point out that the non-local linear part is small, since 
by \eqref{inter2}
$$\|\mathcal L_1(\Phi)\|_{*}\le c e^{-(1-\alpha)\frac{2\pi\rho}{dk}}\|\Phi\|_{*}.$$
So our problem reduces to prove that if $\Phi\in H_\rho^\perp$ solves 
\beq\label{pbintmodello} \mathcal L _0 (\Phi)= h +\mathfrak c\partial_\rho W_\rho\ \hbox{in}\  \mathbb R^n\eeq
for some  $\mathfrak c\in \mathbb R$ then the priori estimate \eqref{norma} holds
and this is done using the same arguments of Lemma 4.3 of \cite{dwy}. For sake of completeness, we give the proof below.\\

First, we prove that
\beq\label{stimac}\left|\mathfrak c\right|\lesssim \left[\left(\frac{1}{\rho^\nu}+ e^{-(1-\alpha)\frac{2\pi}{k d}}\right)\|\Phi\|_{*}+\|h\|_{*}\right].\eeq
Indeed, by  \eqref{pbintmodello} 
$$\int\limits_{\mathbb R^n} \mathcal L_0 (\Phi)\partial_\rho W_\rho -\int\limits_{\mathbb R^n}  h  \partial_\rho W_\rho =\mathfrak c \int\limits_{\mathbb R^n} \left(\partial_\rho W_\rho\right)^2,$$ where arguing as in Lemma \ref{crho}  $$\left|\int\limits_{\mathbb R^n} \mathcal L_0 (\Phi)\partial_\rho W_\rho\right|\lesssim \left(\frac{k}{\rho^\nu}+k e^{-(1-\alpha)\frac{2\pi}{d k}} \right)\|\Phi\|_{*},$$
moreover
$$\left|\int\limits_{\mathbb R^n} h  \partial_\rho W_\rho\right|\lesssim k\| h \|_{*}, \hbox{because}\ |\partial_\rho W_\rho|\lesssim W_\rho$$ and 
 $$\int\limits_{\mathbb R^n}\left(\partial_\rho W_\rho\right)^2\sim c k\ \hbox{for some positive constant}\ c. $$
 
   Next, we show that there exist constants $\tau$ and $c$ (all independent of $k$ ) such that for any $x\in \mathbb R^n\setminus \bigcup_{i=1, \ldots, d\atop j=1, \ldots, k} B(\rho\eta_{ij}, \tau)$ \beq\label{puntuale}\left|\Phi(x)\right|\le C\left(\|\mathcal L_0 (\Phi)\|_{*}+\sup_{i=1, \ldots, d \atop j=1, \ldots, k}\|\Phi\|_{L^\infty(\partial B(\rho\eta_{ij}, \tau))}\right)\sum_{i=1}^d\sum_{j=1}^k e^{-\alpha|x-\rho\eta_{ij}|},\eeq 
   which immediately implies
\beq\label{stima2} \|\Phi\|_{*}\le C\left(\|\mathcal L_0 (\Phi)\|_{*}+\sup_{i=1, \ldots, d\atop j=1, \ldots, k}\|\Phi\|_{L^\infty(\partial B(\rho\eta_{ij}, \tau))}\right).\eeq
To prove the above pointwise estimate we first show the independence of $\tau$ on $k$ for any $x\in \mathbb R^n\setminus \bigcup_{i=1, \ldots, d\atop j=1, \ldots, k} B(\rho\eta_{ij}, \tau)$. Indeed reasoning as in \cite{dwy} and using Lemma 3.4 of \cite{dwy} we get that
$$\begin{aligned} W_\rho(\hat\Theta_i x)&=\sum_{j=1}^k U(x-\rho\eta_{ij})\\
&\le \sum_{|x-\rho\eta_{ij}|<\rho\sin\frac{\pi}{d k}}U(x-\rho\eta_{ij})+\sum_{\ell=1}^{+\infty}\sum_{\ell \rho\sin\frac{\pi}{dk}\le |x-\rho\eta_{ij}|<(\ell+1)\rho\sin\frac{\pi}{dk}}U(x-\rho\eta_{ij})\\
 &\lesssim U(\tau)+c\sum_{\ell=1}^{+\infty}\ell^{n-1}e^{-\ell \rho \frac{\pi}{d k}} \\ & \le C U(\tau). \end{aligned}$$
Thus we can take $\tau$ sufficiently large (but independent of $k$) such that for any $x\in \mathbb R^n\setminus \bigcup_{i=1, \ldots, d\atop j=1, \ldots, k} B(\rho\eta_{ij}, \tau)$
$$3 W_\rho^2(x)\le \frac 12 \frac{V_0-\alpha^2}{4};\quad \beta\sum_{i=2}^d W_\rho^2(\hat\Theta_i x)\le \frac 12\frac{V_0-\alpha^2}{4}. $$ Now we let $\Pi_\pm (x)=\sum_{i=1}^d\sum_{j=1}^k e^{\pm \alpha|x-\rho\eta_{ij}|}.$ For $x\in \mathbb R^n\setminus \bigcup_{i=1, \ldots, d\atop j=1, \ldots, k} B(\rho\eta_{ij}, \tau)$ we get $$\begin{aligned} \mathcal L_0 \left(\Pi_\pm(x)\right)&=\sum_{i=2}^k\sum_{j=1}^ke^{\pm|x-\rho\eta_{ij}|}\left(-\alpha^2\mp\frac{\alpha(n-1)}{|x-\rho\eta_{ij}|}+V(x)-3W_\rho^2(x)-\beta \sum_{i=2}^d W_\rho^2(\hat\Theta_ix)\right)\\
&\ge \sum_{i=2}^k\sum_{j=1}^ke^{\pm|x-\rho\eta_{ij}|} \left(-\alpha^2\mp\frac{\alpha(n-1)}{|x-\rho\eta_{ij}|}+V_0 -\frac{V_0-\alpha^2}{4}\right)\end{aligned}$$ 
Hence
$$\mathcal L_0  (\Pi_\pm (x))\ge c_0 \Pi_\pm (x)$$ for some positive constant $c_0$ independent of $k$.\\ Then we can use $\Pi_\pm(x)$ as barriers and we can apply the maximum principle to the linear operator $\mathcal L_0 $ obtaining 
$$ |\Phi(x)|\le C\left(\|\mathcal L_0 (\Phi)\|_{*}+\sup_{i=1, \ldots, d\atop j=1, \ldots, k}\|\Phi\|_{L^\infty(\partial B(\rho\eta_{ij}, \tau))}\right)\sum_{i=1}^d\sum_{j=1}^k e^{-\alpha|x-\rho\eta_{ij}|}+\delta\sum_{i=1}^d\sum_{j=1}^ke^{\alpha|x-\rho\eta_{ij}|}$$ for any $\delta>0$ and $C$ independent of $k$ and $\delta$. Letting $\delta\to 0$ we get the estimate \eqref{puntuale}.\\

 Finally, we prove \eqref{norma} arguing by contradiction. Assuming that there is a sequence $(\Phi_{\mathtt n},  h_{\mathtt n})$ satisfying \eqref{pbintmodello} such that $$\|\Phi_{\mathtt n}\|_{*}=1\quad \mbox{and}\quad \|h _{\mathtt n}\|_{*}=o(1)\quad\mbox{as}\, k_{\mathtt n}\to+\infty.$$ In the sequel we will omit the dependence on ${\mathtt n}$. By \eqref{stimac} and the fact that  $\| \partial_\rho W_\rho\|_{*}\lesssim 1,$ we also have $\|\mathcal L_0 (\Phi)\|_{*}=o(1)$. Hence \eqref{stima2} implies the existence of a subsequence of $\eta_{ij}$ such that  \beq\label{stima3} \|\Phi\|_{L^\infty(\partial B(\rho\eta_{ij}, \tau))}\ge C>0\eeq for some fixed constant $C$ which is independent of $k$.\\ Since $\|\Phi\|_{L^\infty(\mathbb R^n)}\le 1$ by elliptic regularity estimates we get that $\|\Phi\|_{C^1(\mathbb R^n)}\le C$. By applying the Ascoli-Arzela's Theorem we get the existence of a subsequence of $\eta_{ij}$ such that $\Phi(x+\rho\eta_{ij})$ converges (on compact sets) to $\Phi_\infty$ which is a bounded solution   of $$-\Delta\Phi_\infty+\Phi_\infty -3 U^2\Phi_\infty=0\ \hbox{or}\ -\Delta\Phi_\infty+\Phi_\infty-\beta U^2\Phi_\infty=0.$$ In the first case  $\Phi_\infty\equiv 0$ because $\Phi\in H_\rho^\perp,$ while in the second case $\Phi_\infty\equiv 0$ because $\beta\not =\Lambda_\kappa$. Finally, a contradiction arises because of \eqref{stima3}.

\end{document}